\documentclass[a4paper, 11pt, reqno]{amsart}
\usepackage[english]{babel}
\usepackage[utf8]{inputenc}
\usepackage{amsthm, amsmath, amssymb, amsrefs, enumerate, enumitem, mathtools, mathrsfs}

\usepackage{a4wide}

\usepackage[T1]{fontenc}
\usepackage{lmodern}
\usepackage{microtype}

\usepackage{hyperref}

\theoremstyle{plain}
\newtheorem{thm}{Theorem}[section]
\newtheorem{prop}[thm]{Proposition}
\newtheorem{lemma}[thm]{Lemma}
\theoremstyle{definition}
\newtheorem{defn}[thm]{Definition}
\newtheorem*{exa}{Example}
\theoremstyle{remark}
\newtheorem*{rmk}{Remark}
\newtheorem*{rmks}{Remarks}

\newcommand{\C}{\mathbb{C}}
\newcommand{\Hb}{\mathbb{H}}
\newcommand{\Z}{\mathbb{Z}}
\newcommand{\Q}{\mathbb{Q}}
\newcommand{\N}{\mathbb{N}}
\newcommand{\R}{\mathbb{R}}
\newcommand{\slz}{{\text {\rm SL}}_2(\mathbb{Z})}
\newcommand{\re}{\textnormal{Re}}

\newcommand{\vt}[1]{\left\lvert #1 \right\rvert}
\DeclareMathOperator{\sgn}{sgn}
\newcommand{\dm }{\mathrm{d}}
\newcommand{\Ac}{\mathcal{A}}
\newcommand{\Dc}{\mathcal{D}}

\newcommand{\Fc}{\mathcal{F}}
\newcommand{\Gc}{\mathcal{G}}
\newcommand{\Hc}{\mathcal{H}}
\newcommand{\Kc}{\mathcal{K}}
\newcommand{\Nc}{\mathcal{N}}
\newcommand{\Qc}{\mathcal{Q}}
\newcommand{\Hs}{\mathscr{H}}
\newcommand{\Ks}{\mathscr{K}}
\newcommand{\af}{\mathfrak{a}}
\newcommand{\cf}{\mathfrak{c}}
\newcommand{\ff}{\mathfrak{f}}
\newcommand{\rf}{\mathfrak{r}}
\newcommand{\sfr}{\mathfrak{s}}
\newcommand{\id}{\mathrm{id}}

\numberwithin{equation}{section}
\setlist{nosep}
\setlist{noitemsep}

\makeatletter
\let\@@pmod\pmod
\DeclareRobustCommand{\pmod}{\@ifstar\@pmods\@@pmod}
\def\@pmods#1{\mkern4mu({\operator@font mod}\mkern 6mu#1)}
\makeatother

\title{A modular framework for generalized Hurwitz class numbers I}
\author{Olivia Beckwith}
\address{Department of Mathematics, Tulane University, New Orleans, LA 70118}
\email{obeckwith@tulane.edu}
\author{Andreas Mono}
\address{\textnormal{\textit{Current address:}} Department of Mathematics, Vanderbilt University, 1326 Stevenson Center, Nashville TN 37240, USA \newline \indent \textnormal{\textit{Former address:}} Department of Mathematics and Computer Science, Division of Mathematics, University of Cologne, Weyertal 86--90, 50931 Cologne, Germany}
\email{andreas.mono@vanderbilt.edu}

\begin{document}

\begin{abstract}
We discover a non-trivial relation between the mock modular generating functions of the level $1$ and level $N$ Hurwitz class numbers. This relation yields a holomorphic modular form of weight $\frac{3}{2}$ and level $4N$, where $N > 1$ is stipulated to be odd and square-free. We extend this observation to a non-holomorphic framework and obtain a higher level non-holomorphic Zagier Eisenstein series as well as a preimage $\Gc$ of it under the differential operator $\xi_{\frac{1}{2}}$. All of these observations are deduced from a more general inspection of a certain weight $\frac{1}{2}$ Maass--Eisenstein series of level $4N$ at its spectral point $s=\frac{3}{4}$. This idea goes back to Duke, Imamo\={g}lu and T\'{o}th in level $4$ and relies on the theory of so-called sesquiharmonic Maass forms. We calculate the Fourier expansion of $\Gc$ and $\xi_{\frac{1}{2}}\Gc$. We conclude by offering examples if $N=5$ or $N=7$ as well as some questions for future work.
\end{abstract}



\thanks{The first author was partially supported by Simons Foundation Collaboration Grant \#953473 and National Science Foundation Grant DMS-2401356. Parts of the second author's research on this project was supported by the CRC/TRR 191 ``Symplectic Structures in Geometry, Algebra and Dynamics'', funded by the DFG (project number 281071066).}

\maketitle

\section{Introduction and statement of results}

\subsection{Hurwitz class numbers - holomorphic setting}

Since Shimura's seminal paper \cite{shimura75annals} laid the foundations of the theory of half integral weight modular forms, a natural goal in the field has been finding explicit analytic constructions. Aside from half integral weight holomorphic examples constructed using the Jacobi theta function $\theta$ (see \eqref{eq:thetadef}) or the Dedekind eta function, one of the earliest examples are the Cohen--Eisenstein \cite{cohen75} series of weights greater than $\frac{3}{2}$, whose coefficients are Dirichlet $L$-values $L(s,\chi_D)$ at integral values of $s$. In weight $\frac{3}{2}$, Cohen's coefficients are the Hurwitz class numbers $H(n)$, which generate a so-called mock modular form instead of a modular form (see Subsection \ref{subsec:intrononhol} for a description of mock modular forms). The study of Hurwitz class numbers can be traced back to Gauss, who observed the connection with the sums of three squares problem. To be more specific, we have
\begin{equation}\label{eq:threesquares}
\theta^3(\tau) = 12 \sum_{ n \ge 0} (H(4n) - 2 H(n)) q^n, \qquad q \coloneqq e^{2\pi i \tau}, \qquad \tau \in \Hb \coloneqq \left\{\tau = u+iv \colon v > 0\right\}.
\end{equation}
Since Gauss, Hurwitz class numbers have found numerous applications throughout number theory. Hurwitz himself showed that they admit a class number formula paralleling Dirichlet's famous formula, which was recently extended to the case of indefinite quadratic forms \cite{dit21jems}. Their mock modular properties give rise to recurrence formulas which have been extensively studied \cites{bruschw, mertens16, mertens14, murakami}. Their divisibility properties have been studied in a handful of recent papers \cites{brika2, BOR1, BOR2, BOR3, beckwith2017} and used to study class numbers of imaginary quadratic number fields. They also are known to count isomorphism classes of elliptic curves over finite fields with specified numbers of points \cite{schoof87}, which was recently used to compute distributions of Gaussian hypergeometric functions \cite{ono-saad-saikia}. Other recent appearances of the Hurwitz class numbers involve connections to formulas involving real quadratic number fields \cite{lischw} and shifted convolution $L$-series \cite{walker}. 

The formula \eqref{eq:threesquares} can be viewed as a modular modification of the mock modular Hurwitz generating function. Modular forms similar to \eqref{eq:threesquares} were constructed in \cite{brika1}*{Lemma 2.6}. However, this approach leads to a twist by a non-trivial character and a change of the level of the form. Independent from that, \cite{beckwith2017} studied modular Hurwitz generating functions with indices restricted to given square classes using quadratic twisting arguments.

Cohen's work was extended in level aspect by Pei and Wang \cite{peiwang} in $2003$, who constructed bases for the Eisenstein plus subspace of weight $\frac{3}{2}$ on $\Gamma_0(4N)$ with $N$ odd and square-free. To describe this briefly, we let $\ell \mid N$ and define the \emph{generalized Cohen--Eisenstein series}
\begin{align} \label{eq:Hsdef}
\Hs_{\ell,N}(\tau) \coloneqq \sum_{n \geq 0} H_{\ell,N}(n) q^n,
\end{align}
where we refer to the numbers $H_{\ell,N}$ as \emph{generalized Hurwitz class numbers}. Due to their technical nature, we define these numbers in \eqref{eq:HNNdef} and \eqref{eq:HellNdef} below. As the terminology suggests, these numbers are higher level analogs of the classical Hurwitz class numbers. This is justified by a result from part II of this series \cite{bemo2}*{Theorem 1.4, 6.1 (i)}, where we show that
\begin{align*}
\sum_{Q \in \Qc_{p,n} \slash \Gamma_0(p)} \frac{1}{\vt{\Gamma_0(p)_Q}} &= 4H_{1,1}(-n) - 2H_{p,p}(-n)
\end{align*}
whenever $n < 0$ with $n \equiv 0,1 \pmod*{4}$ and $N=p$ is an odd prime. Here, we use
\begin{align*}
\Qc_{N,n} &\coloneqq \left\{ax^2+bxy+cy^2 \colon a,b,c \in \Z, \ N \mid a, \ b^2-4ac=n \right\}
\end{align*}
for such $n$. In addition, we have $H(n) = H_{1,1}(n)$ and the generalized Hurwitz class numbers appeared in \cites{luozhou, mmrw} very recently.

The extra parameters $\ell$ and $N$ allow a more refined inspection of the non-modularity of the Hurwitz class number generating function $\Hs_{1,1}$. More precisely, it turns out that $\Hs_{1,N}$ is not modular like $\Hs_{1,1}$. However, the functions $\Hs_{\ell,N}$ for $1 < \ell \mid N$ are modular and were the main objects of study in \cites{peiwang, luozhou}. In view of the non-modularity of both $\Hs_{1,1}$ and $\Hs_{1,N}$, we offer the following result.

\begin{thm} \label{thm:mainholomorphic}
Let $N > 1$ be odd and square-free. Then, the linear combination
\begin{align*}
\sum_{n \geq 0} H_{1,1}(n)q^n  - \Big(\prod_{\substack{p \text{ prime} \\ p \mid N}} \frac{p+1}{p} \Big) \sum_{n \geq 1} H_{1,N}(n)q^n
\end{align*}
is a modular form of weight $\frac{3}{2}$ on $\Gamma_0(4N)$. If $N=p$ is an odd prime then
\begin{align*}
\sum_{n \geq 0} H_{1,1}(n)q^n  - \frac{p+1}{p} \sum_{n \geq 1} H_{1,p}(n)q^n = \frac{1}{1-p} \Hs_{p,p}(\tau).
\end{align*}
\end{thm}

At the end of this introduction, we offer two examples which relate the coefficients of $\Hs_{N,N}$ to certain theta functions for $N=5$ or $N=7$ along the lines of \eqref{eq:threesquares}.

\begin{rmks}
\
\begin{enumerate}
\item The modular forms from Theorem \ref{thm:mainholomorphic} are contained in different spaces than those considered in \cite{beckwith2017}, where the level was square. 
\item The character for our function is trivial and hence these modular forms differ from those in \cite{brika1}, where the character is always nontrivial. 
\item Jagathesan and Manickam \cite{jagman} constructed Cohen--Eisenstein series of level $4N$ for any odd $N \in \N$.
\item Theorem \ref{thm:mainholomorphic} is complementary to \cite{luozhou}*{Theorem 7.2}, where the modular forms $\Hs_{\ell,N}$ for $1 < \ell \mid N$ were used to construct another basis for the Eisenstein plus subspace of weight $\frac{3}{2}$ on $\Gamma_0(4N)$.
\end{enumerate}
\end{rmks}

\subsection{Hurwitz class numbers - non-holomorphic setting} \label{subsec:intrononhol}

In $1975$, Zagier \cites{zagier75, zagier76} discovered the weight $\frac{3}{2}$ non-holomorphic modular completion of $\Hs_{1,1}$ on $\Gamma_0(4)$, which is explicitly given by
\begin{align} \label{eq:Hcdef}
\Hc(\tau) \coloneqq -\frac{1}{12} + \sum_{n \geq 1} H_{1,1}(n) q^n + \frac{1}{8\pi\sqrt{v}} + \frac{1}{4\sqrt{\pi}} \sum_{n \geq 1} n \Gamma\left(-\frac{1}{2},4\pi n^2 v\right) q^{-n^2}.
\end{align}
Here and throughout, $\Gamma(s,y)$ denotes the incomplete $\Gamma$-function, see \eqref{eq:incompleteGammaDef} below.

In the early $2000$'s, Bruinier and Funke \cite{brufu02} introduced the notion of \emph{harmonic Maass forms}, which are integral or half integral weight non-holomorphic modular forms whose Fourier expansion splits in a similar way as $\Hc$. Around the same time, Zwegers \cite{zwegers} demonstrated that such forms provided the analytical framework for Ramanujan's elusive mock theta functions. In turn, this observation coined the notion of a \emph{mock modular form}, which is the holomorphic part of a non-trivial harmonic Maass form. Our methods rely on the theory of mock modular forms, particularly, their construction from Maass--Eisenstein series and the study of their behavior under the $\xi_{k}$-operator
\begin{align*}
\xi_k \coloneqq 2i v^k \overline{\frac{\partial}{\partial\overline{\tau}}} = i v^k \overline{ \left( \frac{\partial}{\partial u} + i \frac{ \partial}{\partial v} \right)}
\end{align*}
of Bruinier and Funke. The $\xi_k$-operator plays a fundamental role in the theory of harmonic Maass forms, as it defines a surjective map from spaces of harmonic Maass forms to spaces of weakly holomorpic modular forms. The problem of constructing preimages under $\xi_k$ of a given (weakly) holomorphic modular form is generally difficult and only understood in a few cases explicitly, such as Poincar{\'e} and Eisenstein series \cites{wag, bringmann-ono07}, $\theta^3$ from \eqref{eq:threesquares} \cite{rhowal}, and weight $2$ newforms \cites{agor15}.

In today's terminology, $\Hc$ is a weight $\frac{3}{2}$ harmonic Maass form on $\Gamma_0(4)$. We find the following higher level analog of Zagier's result.
\begin{thm} \label{thm:mainnonholomorphic}
Let $N > 1$ be odd and square-free. Then the function
\begin{align*}
\sum_{n \geq 1} H_{1,N}(n)q^n + \Big(\prod_{\substack{p \text{ prime} \\ p \mid N}} \frac{p}{p+1}\Big) \Bigg(\frac{1}{8\pi\sqrt{v}} + \frac{1}{4\sqrt{\pi}} \sum_{n \geq 1} n \Gamma\left(-\frac{1}{2},4\pi n^2 v\right) q^{-n^2}\Bigg)
\end{align*}
is a weight $\frac{3}{2}$ harmonic Maass form on $\Gamma_0(4N)$. Its image under $\xi_{\frac{3}{2}}$ equals
\begin{align*}
-\frac{1}{16\pi}\Big(\prod_{\substack{p \text{ prime} \\ p \mid N}} \frac{p}{p+1}\Big) \theta(\tau).
\end{align*}
\end{thm}

\begin{rmk}
The modularity of the form in Theorem \ref{thm:mainholomorphic} implies the modularity of the form in Theorem \ref{thm:mainnonholomorphic} and vice versa.
\end{rmk}

Since harmonic Maass forms might be viewed as non-holomorphic preimages of modular forms under $\xi_{k}$, it is natural to ask for preimages of harmonic Maass forms under $\xi_k$, which often are called \emph{sesquiharmonic Maass forms}. A first example of a sesquiharmonic Maass form appeared in a paper by Duke, Imamo\={g}lu and T\'{o}th \cite{dit11annals}, who constructed such a preimage of $\Hc$. More recently, they also appeared in Maass--Poincar{\'e} series whose coefficients involve non-critical modular $L$-values \cite{brdira} and Shintani lifts of weight $0$ harmonic Maass forms \cite{alneschw21}. Iterating the search for preimages under $\xi_k$ leads to so-called \emph{polyharmonic Maass forms} studied in \cites{brika20, lagrho, anlagrho} for integral weights and in \cites{mat19, mat20, alansa} for half integral weights. 

Returning to Duke, Imamo\={g}lu and T\'{o}th's preimage of $\Hc$, we note that it is in fact the ``Eisenstein case'' of their more general construction of weight $\frac{1}{2}$ mock modular forms with Fourier coefficients given by traces of cycle integrals of the modular $j$-function. Their work inspired Rhoades' and Waldherr's \cite{rhowal} construction of a harmonic preimage of $\theta^3$ under $\xi_k$ as well as Wagner's \cite{wag} harmonic preimages of the classical Cohen--Eisenstein series under $\xi_k$. Both papers were generalized in \cite{bss}.

The body of this paper is devoted to the inspection of the higher level ``Eisenstein case'' of Duke, Imamo\={g}lu and T\'{o}th's sesquiharmonic forms. To this end, we define ($\N = \{1,2,3,\ldots\}$)
\begin{align}
\Fc_{k,4N}(\tau,s) &\coloneqq \sum_{\gamma \in \Gamma_0(4N)_{\infty} \backslash \Gamma_0(4N)} v^{s-\frac{k}{2}} \Big\vert_k\gamma, \qquad k \in -\N+\frac{3}{2}, \qquad \re(s) > 1-\frac{k}{2}, \label{eq:Fdef} \\
\Fc_{k,4N}^+(\tau,s) &\coloneqq \mathrm{pr}^+\Fc_{k,4N}(\tau,s),
\end{align}
where $\cdot\vert_k\cdot$ denotes the Petersson slash operator, $\mathrm{pr}^+$ denotes the projection operator into Kohnen's plus space, and $\Gamma_0(4N)_{\infty}$ is the stabilizer of the cusp $i\infty$ (see Section \ref{sec:prelim} for explicit definitions). In Subsection \ref{subsec:Klostermanzetaresidues}, we show that the square-indexed Fourier coefficients of $\Fc_{k,4N}^+$ have poles at $s = \frac{3}{4}$ and compute their residues. By examining all Fourier coefficients of $\Fc_{k,4N}^+$, we establish that 
\begin{align}
\Gc(\tau) \coloneqq \frac{\partial}{\partial s} \left[ \left(s-\frac{3}{4}\right)\Fc_{\frac{1}{2},4N}^+(\tau,s)\right] \Big\vert_{s=\frac{3}{4}} \label{eq:Gdef}
\end{align}
exists in Subsection \ref{subsec:analyticcontinuation}. Theorem \ref{thm:mainholomorphic} as well as Theorem \ref{thm:mainnonholomorphic} follow from the following result, where we refer to \eqref{eq:alphadef} for the definition of Duke, Imamo\={g}lu and T\'{o}th's special function $\alpha$ (slightly renormalized here).
\begin{thm} \label{thm:maingeneral}
Let $N \in \N$ be odd and square-free. 
\begin{enumerate}[label={\rm (\roman*)}]
\item The function $\Gc$ is a weight $\frac{1}{2}$ sesquiharmonic Maass form on $\Gamma_0(4N)$. 
\item Let $\cf(n)$ be given by \eqref{eq:KloostermanZetaResidueRegular}, \eqref{eq:KloostermanZetaResidue0}, \eqref{eq:KloostermanZetaResidueSquare} and $\gamma$ be the Euler--Mascheroni constant. Then, the Fourier expansion of the function $\Gc$ is given by
\begin{multline*}
\hspace*{\leftmargini} \Gc(\tau) = \frac{2}{3} v^{\frac{1}{2}} + \frac{1}{\pi} \Big(\prod_{\substack{p \text{ prime} \\ p \mid N}} \frac{1}{p+1} \Big) \Bigg(\sum_{m \geq 1 }  \left(\gamma + \log\left(\pi m^2\right) +  \alpha\left(4m^2v\right) \right) q^{m^2} - \frac{\log(16v)}{2} \Bigg) \\
+ \frac{2}{3}(1-i) \pi \sum_{\substack{n \geq 0 \\ n \equiv 0,1 \pmod*{4}}} \cf(n)q^n + \frac{2}{3} (1-i) \pi^{\frac{1}{2}} \sum_{\substack{n < 0 \\ n \equiv 0,1 \pmod*{4}}}  \cf(n) \Gamma\left(\frac{1}{2},4\pi\vt{n}v\right) q^n.
\end{multline*}
\item If $N > 1$ then
\begin{align*}
\hspace*{\leftmargini} \frac{1}{4}\xi_{\frac{1}{2}} \Gc(\tau) = \Big(\prod_{\substack{p \text{ prime} \\ p \mid N}} \frac{1}{p+1}\Big) \left(\Hs_{1,1}(\tau) - \Hc(\tau)\right) - \frac{1}{N} \sum_{\ell \mid N} \ell \Big(\prod_{\substack{p \text{ prime} \\ p \mid \ell}} \frac{1}{1-p}\Big) \Hs_{\ell,N}(\tau),
\end{align*}
and this is a weight $\frac{3}{2}$ harmonic Maass form on $\Gamma_0(4N)$. Moreover, we have
\begin{align*}
\Delta_{\frac{1}{2}} \Gc(\tau) = -\frac{1}{4 \pi} \Big(\prod_{\substack{p \text{ prime} \\ p \mid N}} \frac{1}{p+1}\Big) \theta(\tau),
\end{align*}
where $\Delta_{\frac{1}{2}}$ is the weight $\frac{1}{2}$ hyperbolic Laplace operator defined in \eqref{eq:Deltasplitting}.
\end{enumerate}
\end{thm}

\begin{rmks}
\
\begin{enumerate}
\item An alternative approach to constructing modular forms of higher levels as in our theorems is given by the method of theta lifts (see for example \cites{brufu02, brufu06, funke, brufuim}). We explore this approach, and its relationship with the Maass--Eisenstein construction given in this paper, in part II \cite{bemo2}. To be more specific, we connect $\Gc$ to the plus space projection of a regularized Siegel theta lift of the modular function $1$ whenever $N=p$ is an odd prime, which builds on earlier work by Bruinier, Funke, and Imamo\={g}lu \cite{brufuim}. Hence, all the coefficients $\cf(n)$ evaluate as (regularized) quadratic traces of $1$ for such $N$. In particular, the coefficients $\cf(n)$ contain arithmetic information in this case.
\item The higher weight analog of Theorem \ref{thm:maingeneral} as well as its connection to theta lifts whenever $N=p$ is an odd prime can be found in \cite{mo25}.
\item The phenomenon that coefficients of weight $\frac{1}{2}$ mock modular forms are linked to real quadratic fields was explored in further work by Duke, Imamo\={g}lu and T\'{o}th \cite{dit11imrn}.
\end{enumerate}
\end{rmks}

\subsection{Examples}

We conclude by offering two examples of Theorem \ref{thm:mainholomorphic}. The coefficients can be computed using the SAGE \cite{sage} code available as an ancillary file to our arXiv submission.

\begin{exa} \label{exa:examples}
We demonstrate the modularity result from Theorem \ref{thm:mainholomorphic} from a different perspective if the level is prime. This resembles \eqref{eq:threesquares}. Recall Sturm's bound (see \cite{stein}*{Corollary 9.20} for example). We define the ternary quadratic forms
\begin{align*}
Q_5(x,y,z) &\coloneqq 7x^2 + 3 y^2 + 7 z^2 + 2 xy - 6 xz + 2 yz, \\
Q_7(x,y,z) &\coloneqq 4x^2 + 7 y^2 + 8 z^2 - 4 xz.
\end{align*}
\begin{enumerate}
\item If $N = 5$, Sturm's bound implies that
\begin{align*}
3\Hs_{5,5}(\tau) &= -12\Big(\sum_{n \geq 0} H_{1,1}(n)q^n - \frac{6}{5} \sum_{n \geq 1} H_{1,5}(n)q^n\Big) = \sum_{(a,b,c) \in \Z^3} q^{Q_5(a,b,c)} \\
&= 1 + 2q^3 + 6q^7 + 6q^8 + 8q^{12} + 6q^{15} + O\left(q^{20}\right).
\end{align*}
The coefficients can be found in the OEIS \cite{oeis}*{A028973}.
\item If $N = 7$, Sturm's bound implies that
\begin{align*}
2\Hs_{7,7}(\tau) &= -12\Big(\sum_{n \geq 0} H_{1,1}(n)q^n - \frac{8}{7} \sum_{n \geq 1} H_{1,7}(n)q^n\Big) = \sum_{(a,b,c) \in \Z^3} q^{Q_7(a,b,c)} \\
&= 1 + 2q^4 + 2q^7 + 4q^8 + 4q^{11} + 8q^{15} + 6q^{16} + O\left(q^{20}\right).
\end{align*}
\end{enumerate}
\end{exa}

\subsection{Outline of the paper}
The paper is organized as follows. We summarize the required concepts in Section \ref{sec:prelim} and provide the general Fourier expansion of $\Fc_{k,4N}^+$. In Section \ref{sec:Kloostermanzeta}, we move to a closer inspection of the plus space Kloosterman zeta function involved in that Fourier expansion and specialize it to weight $k=\frac{1}{2}$ and $s=\frac{3}{4}$ subsequently. The purpose of this section is to provide explicit formulas and to compute the residue of the plus space Kloosterman zeta function at this spectral point. The content of Section \ref{sec:proofof1.3part1and2} is the proof of Theorem \ref{thm:maingeneral} (i) and (ii). Section \ref{sec:proofof1.3part3} is devoted to the proof of Theorem \ref{thm:maingeneral} (iii), while Section \ref{sec:proofof1.1and1.2} establishes Theorems \ref{thm:mainholomorphic} and \ref{thm:mainnonholomorphic}. We conclude the paper with some questions for future work.

\section*{Acknowledgements}
We would like to thank Scott Ahlgren, Kathrin Bringmann, Nikolaos Diamantis, Michael Griffin, Toshiki Matsusaka and Larry Rolen for helpful remarks on an earlier version of this paper. Moreover, we would like to thank the anonymous referee for their comments and suggestions.

\section*{Notation}
\subsection*{General modular forms notation}
\begin{itemize}
\item $\tau = u+iv \in \Hb$, $q = e^{2\pi i \tau}$,
\item $s$, $\sfr$ are spectral parameters with $\re(s) > 1-\frac{k}{2}$ and related by $\sfr = 2s-\frac{1}{2}$,
\item $N \in \N$ is an odd and square-free level, $\Nc \in \N$ is a general level with $4 \mid \Nc$,
\item $\ell$ is a positive divisor of $N$,
\item $\gamma$ is either the Euler--Mascheroni constant or a $2\times2$ matrix,
\item $\Gamma(s,z)$ is the (analytic continuation of the) principal branch of the incomplete $\Gamma$-function (as a function of $s$), $\Gamma(s) = \Gamma(s,0)$ is the $\Gamma$-function,
\item $\Gamma_0(\Nc)$ is the Hecke congruence subgroup of $\slz$,
\item $\Gamma_0(\Nc)_{\infty} = \left\{\pm \left(\begin{smallmatrix} 1 & n \\ 0 & 1 \end{smallmatrix}\right) \colon n \in \Z \right\} \leq \Gamma_0(\Nc)$ is the stabilizer of the cusp $i\infty$,
\item $\nu_p(n) \in \N_0 \cup \{\infty\}$ is the $p$-adic valuation of $n$,
\item $\left(\frac{a}{b}\right)$ is the Kronecker symbol,
\item $\chi_{d}(n) = \left(\frac{d}{n}\right)$ is a Dirichlet character,
\item $\id = \chi_{1}$ is the principal character of modulus $1$,
\item $\varepsilon_d$ equals $1$ if $d \equiv 1 \pmod*{4}$ and equals $i$ if $d \equiv 3 \pmod*{4}$,
\item $L(s,\chi)$ is the Dirichlet $L$-function and $\zeta(s)$ is the Riemann zeta function,
\item $L_N(s,\chi)$ the incomplete Dirichlet $L$-function defined in \eqref{eq:L_Ndef},
\item $\mu(n)$ the M{\"o}bius function,
\item $\sigma_{N,s}(n)$ and $\sigma_{\ell,N,s}(n)$, $N \neq \ell \mid N$, are certain modifications of the standard sum of powers of divisors function (see Section \ref{sec:prelim}),
\item $M_{k}(\Nc)$ is the $\C$-vector space of holomorphic modular forms of weight $k$ on $\Gamma_0(\Nc)$,
\item $M_{k}^+(\Nc) \subseteq M_{k}(\Nc)$ is Kohnen's plus subspace inside $M_{k}(\Nc)$,
\item $E_{k}^+(\Nc) \subseteq M_{k}^+(\Nc)$ is the Eisenstein subspace inside $M_{k}^+(\Nc)$,
\item $M_{\mu,\nu}(y)$ and $W_{\mu,\nu}(y)$ are the usual $M$- and $W$-Whittaker functions,
\item $\af \in \Q \cup \{i\infty\}$ is a cusp,
\item $\xi_k = 2iv^k\overline{\frac{\partial}{\partial\overline{\tau}}}$,
\item $\Delta_k = -\xi_{2-k}\xi_k$ is the hyperbolic Laplace operator (see \eqref{eq:Deltasplitting}),
\item $\cdot \vert_{k}\cdot$ is the Petersson slash operator (see Section \ref{sec:prelim}),
\item $\cdot\vert\mathrm{pr}^+$ is the projection operator into Kohnen's plus space (see Section \ref{sec:prelim}).
\end{itemize}

\subsection*{Functions and class numbers}
\begin{itemize}
\item $\Fc_{k,4N}(\tau,s)$ is the weight $k \in -\N+\frac{3}{2}$ and level $4N$ Maass--Eisenstein series with spectral parameter $s \in \C$ satisfying $\re(s) > 1-\frac{k}{2}$, see \eqref{eq:Fdef}
\item $\Fc_{k,4N}^+(\tau,s)$ is the projection of $\Fc_{k,4N}(\tau,s)$ to Kohnen's plus space, see \eqref{eq:DeltaG}
\item $\Gc(\tau)$ is the constant coefficient in the Laurent expansion of $\Fc_{\frac{1}{2},4N}^+(\tau,s)$ about $s=\frac{3}{4}$, see \eqref{eq:Gdef},
\item $H_{\ell,N}(n)$, $\ell \mid N$, are the generalized Hurwitz class numbers, see \eqref{eq:HNNdef}, \eqref{eq:HellNdef}, 
\item $H(n) = H_{1,1}(n)$ are the classical Hurwitz class numbers,
\item $\Hs_{\ell,N}(\tau)$, where $\ell \mid N$, are the generating functions of $H_{\ell,N}(n)$; if $N > 1$ is odd and square-free and $\ell > 1$ then $\Hs_{\ell,N} \in M_{\frac{3}{2}}^+(4N)$,
\item $\Hc(\tau)$ is Zagier's non-holomorphic Eisenstein series, see \eqref{eq:Hcdef},
\item $\theta(\tau) = \sum_{n \in \Z} q^{n^2} \in M_\frac{1}{2}(4)$ is the classical Jacobi theta function.
\end{itemize}

\subsection*{Special notation}
\begin{itemize}
\item $K_k(m,n;c)$ is the classical Kloosterman sum with weight $k \in \Z+\frac{1}{2}$ theta multiplier system normalized by $\frac{1}{c}$ (see Section \ref{sec:prelim}),
\item $\Kc_{-\rho,k}$ is Kohnen's modified Kloosterman sum normalized by $\frac{(-1)^{k-\frac{1}{2}}i}{4}$ (see Section \ref{sec:prelim}),
\item $\Ks_{k,4N}^+(m,n;s)$ is the plus space Kloosterman zeta function defined in \eqref{eq:Kloostermanzetadefinition},
\item $\alpha(y)$ is $4\pi$ times Duke, Imamo\={g}lu and T\'{o}th's \cite{dit11annals}*{p.\ 952} special function and defined in \eqref{eq:alphadef},
\item $\rf(N) = \frac{1}{2\pi} \prod_{p \text{ prime}, \ p\mid N} \frac{1}{p+1}$ is the multiple of $\theta$ in the resiude of $\Fc_{\frac{1}{2},4N}^+$ at $s=\frac{3}{4}$, see \eqref{eq:resiudemultiple} and Proposition \ref{prop:Eisensteinresidue},
\item $\cf(n)$ are Fourier coefficients of $\Gc$, see \eqref{eq:cfdef}, \eqref{eq:KloostermanZetaResidueRegular}, \eqref{eq:KloostermanZetaResidue0} and \eqref{eq:KloostermanZetaResidueSquare},
\item $a\left(p^j,n\right)$ are certain Gauss sums pertaining to weight $\frac{1}{2}$, see Lemma \ref{lem:maasseval1},
\item $A(p,n)$ are local factors pertaining to weight $\frac{1}{2}$, see before Proposition \ref{prop:localfactors},
\item $\Ac_{r}(p,n)$ are Pei's and Wang's variant of $A(p,n)$ pertaining to weights $r + \frac{1}{2}$, $r \in \N$, see Section \ref{sec:proofof1.3part3},
\item $T_{\Nc,\sfr}^{\chi}$ is a certain divisor-type sum introduced before Lemma \ref{lem:TdivisorSum} arising from the square part of the Fourier index of $\Fc_{\frac{1}{2},4N}^+$,
\item $f_j(\tau)$ are the coefficients in the Laurent expansion of $\Fc_{\frac{1}{2},4N}^+(\tau,s)$ about $s=\frac{3}{4}$ introduced in \eqref{eq:laurentformal},
\item $g(\tau) \in M_{\frac{3}{2}}^+(4N)$ is introduced in \eqref{eq:gunknown}, see Proposition \ref{prop:Gshadow} and \eqref{eq:peiandwangresult} too.
\end{itemize}

\section{Preliminaries} \label{sec:prelim}

\subsection{Holomorphic and non-holomorphic modular forms}
Let $k \in \frac{1}{2}\Z$. We choose the principal branch of the square-root throughout. The \emph{(Petersson) slash operator} is defined as 
\begin{align*}
\left(f\vert_k\gamma\right)(\tau) \coloneqq \begin{cases}
(c\tau+d)^{-k} f(\gamma\tau) & \text{if } k \in \Z, \\
\left(\frac{c}{d}\right)\varepsilon_d^{2k}(c\tau+d)^{-k} f(\gamma\tau) & \text{if } k \in \frac{1}{2}+\Z,
\end{cases}
\quad \gamma = \left(\begin{matrix} a & b \\ c& d \end{matrix}\right) \in \begin{cases}
\slz & \text{if } k \in \Z, \\
\Gamma_0(4) & k \in \frac{1}{2}+\Z,
\end{cases}
\end{align*}
where $\left(\frac{c}{d}\right)$ denotes the Kronecker symbol, and $\varepsilon_d \coloneqq 1,i$ if $d \equiv \pm1 \pmod*{4}$ for odd integers $d$ (this is guaranteed whenever $\gamma \in \Gamma_0(4)$). The \emph{weight $k$ hyperbolic Laplace operator} is given by
\begin{align} \label{eq:Deltasplitting}
\Delta_k \coloneqq -v^2\left(\frac{\partial^2}{\partial u^2}+\frac{\partial^2}{\partial v^2}\right) + ikv\left(\frac{\partial}{\partial u} + i\frac{\partial}{\partial v}\right) = - \xi_{2-k} \xi_{k}.
\end{align}
This established, we define various classes of modular objects. 
\begin{defn}
Let $k \in \frac{1}{2}\Z$ and $\Nc \in \N$. Suppose that $4 \mid \Nc$ whenever $k \in \frac{1}{2} + \Z$. Let $f \colon \Hb \to \C$ be a smooth function.
\begin{enumerate}[label={\rm (\alph*)}]
\item We call $f$ a \emph{(holomorphic) modular form} of weight $k$ on $\Gamma_0(\Nc)$, if $f$ satisfies the following conditions:
\begin{enumerate}[label={\rm (\roman*)}]
\item For all $\gamma \in \Gamma_0(\Nc)$, we have $f|_k \gamma = f$.
\item The function $f$ is holomorphic on $\Hb$, that is $\xi_k f = 0$.
\item The function $f$ is holomorphic at every cusp of $\Gamma_0(\Nc)$.
\end{enumerate}
\item We call $f$ a \emph{harmonic Maass form of weight $k$} on $\Gamma_0(\Nc)$, if $f$ satisfies the following conditions:
\begin{enumerate}[label={\rm  (\roman*)}]
\item For all $\gamma \in \Gamma_0(\Nc)$, we have $f|_k \gamma = f$.
\item The function $f$ is an eigenfunction of $\Delta_{k}$ with eigenvalue $0$, that is $-\xi_{2-k} \xi_k f = 0$.
\item The function $f$ is of at most linear exponential growth towards all cusps of $\Gamma_0(\Nc)$.
\end{enumerate}
\item We call $f$ a \emph{sesquiharmonic Maass form of weight $k$} on $\Gamma_0(\Nc)$, if $f$ satisfies the following conditions:
\begin{enumerate}[label={\rm  (\roman*)}]
\item For all $\gamma \in \Gamma_0(\Nc)$, we have $f|_k \gamma = f$.
\item The function $f$ satisfies $\xi_{k} \xi_{2-k} \xi_{k} f = 0$.
\item The function $f$ is of at most linear exponential growth towards all cusps of $\Gamma_0(\Nc)$.
\end{enumerate}
\end{enumerate}
Forms in \emph{Kohnen's plus space} are half integral weight forms, whose Fourier expansions are supported on indices $n$ satisfying $(-1)^{k-\frac{1}{2}} n \equiv 0$, $1 \pmod*{4}$.
\end{defn}

The incomplete Gamma function is defined by
\begin{align} \label{eq:incompleteGammaDef}
\Gamma(s,z) \coloneqq \int_z^{\infty} t^{s-1}\mathrm{e}^{-t} \dm t, \qquad \re(s) > 0, \qquad z \in \C.
\end{align}
It is multi-valued as a function of $s$. If $z \neq 0$, each branch can be analytically continued in $s$ with removable singularities at non-positive integers. We restrict to principal values throughout by imposing that the path of integration has to exclude $\R_{\leq 0}$.  We refer to \cite{brdieh}*{Section 2.2} for more details. We have the asymptotic behavior
\begin{align*}
\Gamma(s,v) \sim v^{s-1}\mathrm{e}^{-v}, \quad \vert v \vert \to \infty 
\end{align*}
for $v \in \R$, see \cite{table}*{item 8.3357} or \cite{nist}*{\S 8.11}. 

Equation \eqref{eq:Deltasplitting} implies that a weight $k \neq 1$ harmonic Maass form $f$ naturally splits into a holomorphic and a non-holomorphic part. Explicitly, we have
\begin{align*}
f(\tau) = \sum_{n \gg -\infty} c_f^+(n) q^n + c_f^-(0)v^{1-k} + \sum_{\substack{n \ll \infty \\ n \neq 0}} c_f^-(n)\Gamma(1-k,-4\pi nv)q^n,
\end{align*}
where the notation $\sum_{n \gg -\infty}$ abbreviates $\sum_{n \geq m_f}$ for some $m_f \in \Z$. The notation $\sum_{n \ll \infty}$ is defined analogously and similar expansions hold at the other cusps. A \emph{mock modular form} of weight $k$ is the holomorphic part of a harmonic Maass form whose non-holomorphic part is not identically zero.

Sesquiharmonic Maass forms possess a third part in their Fourier expansion, see \cite{brika20}*{Lemma 4.1}. In weight $\frac{1}{2}$, this part is given in terms of the special function
\begin{align} \label{eq:alphadef}
\alpha(y) \coloneqq \sqrt{y} \int_{0}^{\infty} \frac{\log(t+1)}{\sqrt{t}} e^{-\pi y t} \dm t, \qquad y > 0,
\end{align}
which is a renormalization of \cite{dit11annals}*{p.\ $952$}. Moreover, we require the $W$-Whittaker function (see \cite{nist}*{\S 13.14} for instance). The $W$-Whittaker function satisfies
\begin{align} \label{eq:whittakerspectral}
\begin{split}
W_{\frac{k}{2},\frac{1}{2}-\frac{k}{2}}(4\pi n v) &= (4\pi n v)^{\frac{k}{2}}e^{-2\pi n v}, \\
W_{-\frac{k}{2},\frac{1}{2}-\frac{k}{2}}(4\pi \vt{n} v) &= (4\pi \vt{n} v)^{\frac{k}{2}}e^{2\pi \vt{n} v}\Gamma(1-k,4\pi\vt{n}v)
\end{split}
\end{align}
according to NIST \cite{nist}*{items 13.18.2, 13.18.5}.
\begin{lemma} \label{lem:preimageofincompletegamma}
Let $n \in \N$. We have
\begin{align*}
\frac{1}{(4\pi n v)^{\frac{1}{4}}} \frac{\partial}{\partial s} W_{\frac{1}{4},s-\frac{1}{2}}(4\pi n v)\Big\vert_{s=\frac{3}{4}} &= e^{-2\pi n v} \alpha(4nv),
\\
\xi_{\frac{1}{2}} \left(\alpha(4nv)q^n\right) &= -\sqrt{\pi n} \Gamma\left(-\frac{1}{2},4\pi nv \right) q^{-n}.
\end{align*}
\end{lemma}

\begin{proof}
The first claim is proven in \cite{alansa}*{Lemma 9}. Following \cite{brdira}*{(4.10)}, define the function 
\begin{align*}
\Gamma_{- \frac{1}{2}} (y) \coloneqq \int_y^{\infty} t^{- \frac{1}{2}} \Gamma\left( - \frac{1}{2}, t\right) e^t  \dm t.
\end{align*}
Utilizing \cite{table}*{(8.853)}, we infer that
\begin{multline*}
\Gamma_{-\frac{1}{2}} (4 \pi n y) 
=  \frac{1}{\Gamma(\frac{3}{2})} \int_0^{\infty}  e^{-u} u^{\frac{1}{2}} \left(\int_{4 \pi n y}^{\infty} t^{-1} (t+u)^{-1} \dm t \right) \dm u \\
=  \frac{1}{\Gamma(\frac{3}{2})} \int_0^{\infty}  e^{-u} u^{\frac{1}{2}} \left(\int_{4 \pi n y}^{\infty} \frac{1}{u} \left(\frac{1}{t}-\frac{1}{t+u}\right) \dm t \right) \dm u \\
=  \frac{1}{\Gamma(\frac{3}{2})} \int_0^{\infty}  e^{-u} u^{-\frac{1}{2}} \log \left(1 + \frac{u}{4 \pi n y} \right) \dm u
= 2 \alpha (4 n y).
\end{multline*}
Now, we can directly compute
\begin{multline*}
\xi_{\frac{1}{2}} \left( q^n \alpha (4 ny) \right) = \overline{q^n} \xi_{\frac{1}{2}} \left( \frac{1}{2} \Gamma_{-\frac{1}{2}} (4 \pi n y)  \right) 
= \overline{q^n} \xi_{\frac{1}{2}} \left( \frac{1}{2} \int_{4 ny} ^{\infty} \Gamma \left( - \frac{1}{2}, t \right) e^t t^{- \frac{1}{2}} \dm t  \right)  \\
= i y^{\frac{1}{2}} \frac{1}{2} \overline{(-iq^n)} (4 \pi n y)^{- \frac{1}{2}} e^{4 \pi n y} \Gamma \left( - \frac{1}{2}, 4 \pi n y \right) (4 \pi n)
=  - q^{-n } (\pi n )^{ \frac{1}{2}}   \Gamma \left( - \frac{1}{2}, 4 \pi n y \right),
\end{multline*}
which completes the proof.
\end{proof}

More details on modular forms, harmonic and sesquiharmonic Maass forms can be found in \cite{thebook} for example. 

\subsection{Work of Kohnen}

Following Kohnen \cite{koh85}*{p.\ 254, 256}, we define the Kloosterman sum
\begin{align*}
K_{k}(m,n;c) \coloneqq \frac{1}{c}\sum_{\substack{r \pmod*{c} \\ \gcd(c,r)=1}} \left(\frac{c}{r}\right)\varepsilon_r^{2k} e^{2\pi i \frac{mr^*+nr}{c}}, \quad k \in \frac{1}{2}+\Z, \quad rr^* \equiv 1 \pmod*{c}, \quad 4 \mid c,
\end{align*}
as well as the modified Kloosterman sum ($rr^* \equiv 1 \pmod*{Nc}$)
\begin{multline*}
\Kc_{-\rho,k}(m,n;Nc) \coloneqq 
\frac{(-1)^{k-\frac{1}{2}}i}{4} \left(\frac{4}{-Nc}\right)\varepsilon_{Nc}^{-2k} \sum_{\substack{r \pmod*{Nc} \\ \gcd(Nc,r) = 1}} \left(\frac{r}{Nc}\right) e^{2\pi i \frac{mr^*+nr}{Nc}}, \qquad k \in \frac{1}{2}+\Z, 
\end{multline*}
where the latter is associated to the cusp $\frac{-\rho}{N}$ with $\rho \coloneqq \left(\frac{-4}{N}\right) \in \{\pm 1\}$. In Kohnen's notation, we have
\begin{align*}
K_{k}(m,n;4Nc) = \mathscr{H}_{4Nc}(n,m), \qquad \Kc_{-\rho,k}(m,n;Nc) = \frac{(-1)^{k-\frac{1}{2}}i}{4}\mathscr{H}'_{Nc}(n,m).
\end{align*}
Furthermore, note that it suffices to inspect $K_{\frac{1}{2}}$ and $K_{\frac{3}{2}}$ by definition of $\varepsilon_r$. In turn, both Kloosterman sums are related by
\begin{align} \label{eq:Kloostermanfunctionalequation}
K_{\frac{3}{2}}(m,n;c) = -iK_{\frac{1}{2}}(-m,-n;c).
\end{align}
In addition, Kohnen proved the following identities.
\begin{lemma}[\protect{\cite{koh85}*{p.\ 257}}] \label{lem:kohnenkloostermanidentity}
Suppose that $c\equiv 1 \pmod*{2}$. Then, we have
\begin{align*}
K_{k}(m,n;4Nc) = \begin{cases}
\left(1-(-1)^{k-\frac{1}{2}}i\right)\Kc_{-\rho,k}\left(4^*m,\frac{n}{4};Nc\right) & \text{if } n \equiv 0 \pmod*{4}, \\
\frac{1}{\sqrt{2}}\left(\frac{(-1)^{k-\frac{1}{2}}n}{2}\right)K_{k}(4m,n;8Nc) & \text{if } (-1)^{k-\frac{1}{2}}n \equiv 1 \pmod*{4}.
\end{cases}
\end{align*}
\end{lemma}

Furthermore, we need the projection operator $\mathrm{pr}^{+}$ of a half integral weight form into the Kohnen plus space. Kohnen \cite{koh85}*{pp. 250} constructs this operator explicitly and he computes the right hand side in terms of the Fourier expansions at the cusps $i\infty$, $\frac{1}{2N}$, and $\frac{-\rho}{N}$, where $\rho \coloneqq \left(\frac{-4}{N}\right) \in \{\pm 1\}$ (see \cite{koh85}*{Proposition 3}) . The scaling matrices are
\begin{align*}
\left(\begin{matrix} 1 & 0 \\ 0 & 1\end{matrix}\right), \qquad \left(\begin{matrix} 1 & 0 \\ 2N & 1\end{matrix}\right), \qquad \left(\begin{matrix} 1 & 0 \\ -\rho N & 1\end{matrix}\right),
\end{align*}
which send the cusp $i\infty$ to the aforementioned cusps respectively. The width of $i\infty$ is $h_{i\infty} = 1$, while the widths of the other two cusps both equal $4$. As shown in \cite{thebook}*{Subsection 6.3}, his construction and results on $\mathrm{pr}^{+}$ extend verbatim to the case of half integral weight harmonic Maass forms.

\subsection{Fourier expansions of non-holomorphic Eisenstein series}
We state the following Fourier expansions.
\begin{prop} \label{prop:Fourierexpansionsatcusps}
Let $k \in -\N + \frac{3}{2}$ and $\re(s) > 1-\frac{k}{2}$.
\begin{enumerate}[label={\rm  (\roman*)}]
\item At the cusp $\af=i\infty$, we have 
\begin{multline*}
\hspace*{\leftmargini} \Fc_{k,4N}(\tau,s) = v^{s-\frac{k}{2}} +  \frac{4^{1-s}\pi i^{-k}\Gamma(2s-1)}{\Gamma\left(s+\frac{k}{2}\right)\Gamma\left(s-\frac{k}{2}\right)}\sum_{c \geq 1 } \frac{K_{k}(0,0;4Nc)}{(4Nc)^{2s-1}} v^{1-s-\frac{k}{2}} \\
+ i^{-k}\pi^{s} v^{-\frac{k}{2}}\sum_{n \neq 0} \sum_{c \geq 1} \frac{K_{k}(0,n;4Nc)}{(4Nc)^{2s-1}} \frac{\vt{n}^{s-1} W_{\sgn(n)\frac{k}{2},s-\frac{1}{2}}(4\pi \vt{n} v) e^{2\pi i n u}}{\Gamma\left(s+\sgn(n)\frac{k}{2}\right)} .
\end{multline*}
\item At the cusp $\af=\frac{-\rho}{N}$, $\rho = \left(\frac{-4}{N}\right)$, we have
\begin{multline*}
\hspace*{\leftmargini} \Fc_{k,4N}\Big\vert_{k} \left(\begin{matrix} 1 & 0 \\ -\rho N & 1\end{matrix}\right)(\tau,s)
= \frac{4^{s}\pi i^{-k}\Gamma(2s-1)}{\Gamma\left(s+\frac{k}{2}\right)\Gamma\left(s-\frac{k}{2}\right)}v^{1-s-\frac{k}{2}} \sum_{\substack{c > 0 \\ c \equiv 1 \pmod*{2}}}  \frac{\Kc_{-\rho,k}(0,0;Nc)}{(Nc)^{2s-1}} \\
+ 4^{s}\pi^{s}i^{-k} v^{-\frac{k}{2}} \sum_{n \neq 0} \sum_{\substack{c > 0 \\ c \equiv 1 \pmod*{2}}} \frac{i^{-n}\Kc_{-\rho,k}(0,n;Nc)}{(Nc)^{2s-1}} \frac{\vt{n}^{s-1} W_{\sgn(n)\frac{k}{2},s-\frac{1}{2}}(\pi \vt{n} v) e^{2\pi i n \frac{u}{4}}}{\Gamma\left(s+\sgn(n)\frac{k}{2}\right)}. 
\end{multline*}
\item At the cusp $\af=\frac{1}{2N}$, we have
\begin{multline*}
\hspace*{\leftmargini} \Fc_{k,4N}\Big\vert_{k} \left(\begin{matrix} 1 & 0 \\ 2 N & 1\end{matrix}\right)(\tau,s) = \frac{4^s \pi i^{-k} \Gamma(2s-1)}{\Gamma\left(s+\frac{k}{2}\right)\Gamma\left(s-\frac{k}{2}\right)} v^{1-s-\frac{k}{2}} 
\sum_{\substack{c > 0 \\ c \equiv 1 \pmod*{2}}} \frac{K_{k}(0,0;8Nc)}{(8Nc)^{2s-1}} \\
+ 4^s\pi^s i^{-k} v^{-\frac{k}{2}} \sum_{n \neq 0} \sum_{\substack{c > 0 \\ c \equiv 1 \pmod*{2}}} \frac{K_{k}(0,n;8Nc)}{(8Nc)^{2s-1}}  \frac{\vt{n}^{s-1} W_{\sgn(n)\frac{k}{2},s-\frac{1}{2}}(\pi \vt{n} v) e^{2\pi i n \frac{u}{4}}}{\Gamma\left(s+\sgn(n)\frac{k}{2}\right)}.
\end{multline*}
\end{enumerate}
\end{prop}

\begin{rmks}
\
\begin{enumerate}
\item The first item specializes to \cite{thebook}*{Theorem 6.15 (v)} at $s=1-\frac{k}{2}$ by \eqref{eq:whittakerspectral}.
\item According to \cite{nist}*{item 13.19.3}, we have
\begin{align*}
W_{\lambda,\mu}(x) \sim x^{\lambda}e^{-\frac{1}{2}x}, \qquad x \to \infty,
\end{align*}
which establishes convergence of the Fourier expansions.
\end{enumerate}
\end{rmks}

\begin{proof}[Proof of Proposition \ref{prop:Fourierexpansionsatcusps}]
The proof is standard and follows Kohnen's calculations \cite{koh85}*{pp.\ 252--256}. Since we work in a non-holomorphic setting, we require a generalization of the classical Lipschitz summation formula, which is established by Maass \cite{maass64}*{pp.\ 207} and Siegel \cite{siegel56}*{p.\ 366}. The connection between the integrals there and the $W$-Whittaker functions is given by (see \cite{table}*{9.222} for example)
\begin{align*}
I\left(y,s\pm\frac{k}{2},s\mp\frac{k}{2}\right) = \int_{0}^{\infty}  (t+1)^{s\pm\frac{k}{2}-1} t^{s\mp\frac{k}{2}-1} e^{-ty} \dm t = \frac{\Gamma\left(s\mp\frac{k}{2}\right)}{y^{s}}e^{\frac{y}{2}}W_{\pm\frac{k}{2},s-\frac{1}{2}}(y).
\end{align*}
Combining yields the claimed expansions.
\end{proof}

We define the \emph{plus space Kloosterman zeta function}
\begin{align} \label{eq:Kloostermanzetadefinition}
\begin{split}
\Ks_{k,4N}^+(m,n;s) &\coloneqq \sum_{c > 0} \frac{1+\left(\frac{4}{c}\right)}{(4Nc)^{s}} \sum_{\substack{r \pmod*{4Nc} \\ \gcd(4Nc,r)=1}} \left(\frac{4Nc}{r}\right)\varepsilon_r^{2k} e^{2\pi i \frac{mr^*+nr}{4Nc}}.
\end{split}
\end{align}
The sums $K_{k}(m,n;c)$ do satisfy an analog of Weil's bound for the integral weight Kloosterman sums, see \cite{dit11annals}*{Section 3} for example, which yields absolute convergence of $\Ks_{k,4N}^+(m,n;2s)$ as a function of $s$ for $\re(s) > 1$. The plus space Kloosterman zeta function satisfies a functional equation (see \cite{iwa02}*{Chapter 9} for example), which is motivated by \eqref{eq:Kloostermanfunctionalequation}.  Selberg \cite{selberg65}*{(3.9)}, Goldfeld--Sarnak \cite{gosa}*{Theorem 1} and Goldfeld--Hoffstein \cite{goho}*{Section 3} proved that $\Ks_{k,4N}^+(m,n;2s)$ admits a meromorphic continuation in $s$ to $\re(s) > \frac{1}{2}$ with a finite number of poles inside the real interval $\left(\frac{1}{2},1\right)$. The residues can be found in \cite{ilt22}*{Proposition 2.5} for example. In the same paper, the plus space Kloosterman zeta function yields a certain trace of the Selberg--Poincar{\'e} series \cite{ilt22}*{Theorem 3.1} along the lines of Duke, Imamo\={g}lu and T\'{o}th \cite{dit11annals}*{Proposition $4$}.

Summing up, the Fourier expansion of  $\Fc_{k,4N}^+$ is as follows.
\begin{prop} \label{prop:Eisensteinprojected}
Let $k \in -\N + \frac{3}{2}$ and $\re(s) > 1-\frac{k}{2}$. Then
\begin{multline*}
\Fc_{k,4N}^+(\tau,s) = \frac{2}{3}\left(v^{s-\frac{k}{2}} + \frac{4^{1-s}\pi i^{-k}\Gamma(2s-1)}{\Gamma\left(s+\frac{k}{2}\right)\Gamma\left(s-\frac{k}{2}\right)} \Ks_{k,4N}^+(0,0;2s) v^{1-s-\frac{k}{2}} \right) \\
+ \frac{2}{3}i^{-k}\pi^{s}v^{-\frac{k}{2}} \sum_{\substack{n \neq 0 \\ (-1)^{k-\frac{1}{2}}n \equiv 0,1 \pmod*{4}}}  \frac{\Ks_{k,4N}^+(0,n;2s) \vt{n}^{s-1}}{\Gamma\left(s+\sgn(n)\frac{k}{2}\right)} W_{\sgn(n)\frac{k}{2},s-\frac{1}{2}}(4\pi \vt{n} v) e^{2\pi i n u}.
\end{multline*}
\end{prop}

\begin{rmk}
This result can be found as a part of \cite{jkk1}*{Theorem 4.4} as well.
\end{rmk}

\begin{proof}
Let $c_{\af}(n,v)$ be the $n$-th Fourier coefficient of the Fourier expansion of $\Fc_{k,4N}$ at the cusp $\af$. Kohnen's result \cite{koh85}*{Proposition 3} on the projection operator in terms of Fourier expansions generalizes straightforwardly to harmonic Maass forms by \cite{thebook}*{Proposition 6.7}. We observe that the proof in our framework with an additional spectral parameter is the same except for the fact that our Fourier expansions are given in terms of 
\begin{align*}
W_{\sgn(n)\frac{k}{2},s-\frac{1}{2}}\left(4\pi \vt{n} \frac{v}{h_{\af}}\right) e^{2\pi i n \frac{u}{h_{\af}}}
\end{align*}
instead of $q^{\frac{n}{h_{\af}}}$. In other words, the necessary calculations carry over to our slightly more general Fourier expansion, getting
\begin{multline*}
\Fc_{k,4N}^+(\tau,s) = \frac{2}{3}\sum_{\substack{n \in \Z \\ n \equiv 0 \pmod*{4}}} \left(c_{i\infty}(n,v) + \left(1-(-1)^{k-\frac{1}{2}}i\right)2^{2k-2} i^{\frac{n}{4}} c_{-\frac{\rho}{N}}\left(\frac{n}{4},16v\right)\right)e^{2\pi i n u} \\
+ \frac{2}{3}\sum_{\substack{n \in \Z \\ (-1)^{k-\frac{1}{2}}n \equiv 1 \pmod*{4}}} \left(c_{i\infty}(n,v) + 2^{k-\frac{3}{2}} \left(\frac{(-1)^{k-\frac{1}{2}}n}{2}\right) c_{\frac{1}{2N}}(n,4v) \right)e^{2\pi i n u}.
\end{multline*}
The result follows by inserting the Fourier expansions from Proposition \ref{prop:Fourierexpansionsatcusps} and utilizing Kohnen's identity on Kloosterman sums (see Lemma \ref{lem:kohnenkloostermanidentity}).
\end{proof}

\subsection{Work of Pei and Wang}
Let $\Nc \in \N$, $N > 1$ be odd and square-free, and define
\begin{align}
\chi_{d} &\coloneqq \left(\frac{d}{\cdot}\right),  \nonumber \\
L_{\Nc}(s,\chi) &\coloneqq L(s,\chi)\prod_{\substack{p \text{ prime} \\ p \mid \Nc}} \left(1-\chi(p)p^{-s}\right) = \prod_{\substack{p \text{ prime} \\ p \nmid \Nc}}\frac{1}{1-\chi(p)p^{-s}} = \sum_{\substack{n \geq 1 \\ \gcd(n,\Nc)=1}} \frac{\chi(n)}{n^s}. \label{eq:L_Ndef}
\end{align}
In \cite{peiwang}*{Theorem 1 (I)}, Pei and Wang show that the dimension of $E_{\frac{3}{2}}^+(4N)$ (the Eisenstein plus subspace of $M_{\frac{3}{2}}(4N)$) equals $2^{\omega(N)}-1$. Here, $\omega(N) \coloneqq \sum_{p \text{ prime}, \ p \mid N} 1$ counts the distinct prime divisors of $N$. The idea of their proof is to construct a basis $E_{\frac{3}{2}}^+(4N)$ out of a basis of $E_{\frac{3}{2}}(4N)$. The latter basis is the main result in Pei \cite{peiII}, while the former basis is explicitly given by the generalized Cohen--Eisenstein series $\Hs_{\ell,N}$ from \eqref{eq:Hsdef} with $1 < \ell \mid N$.

We emphasize that there is no generalized Cohen--Eisenstein series for the divisor $\ell=1$ in weight $\frac{3}{2}$. That is, the generating function of $H_{1,N}(n)$ is not modular. Let $\ell \mid N$ and define
\begin{align*}
\sigma_{\ell,N,s}(r) \coloneqq
\sum_{\substack{d \mid r \\ \gcd(d,\ell)=1 \\ \gcd\left(\frac{r}{d},\frac{N}{\ell}\right)=1}} d^s, \qquad \sigma_{N,s}(r) \coloneqq \sigma_{N,N,s}(r).
\end{align*} 
For $\ell \neq N$, the Fourier coefficients in weight $\frac{3}{2}$ are given by
\begin{align} \label{eq:HNNdef}
H_{N,N}(n) 
&\coloneqq \begin{cases}
L_N\left(-1,\mathrm{id}\right) & \text{if } n=0, \\
L_N(0,\chi_t) \sum\limits_{\substack{a \mid m \\ \gcd(a,N) = 1}} \mu(a) \chi_t(a) \sigma_{N,1}\left(\frac{m}{a}\right) & \begin{array}{@{}l} \text{if } -n=tm^2, \\ t \text{ fundamental}, \end{array}  \\
0 & \text{else},
\end{cases}
\end{align}
as well as by
\begin{align} \label{eq:HellNdef}
H_{\ell,N}(n)
&\coloneqq \begin{cases}
0 & \text{if } n=0, \\
L_{\ell}\left(0,\chi_{t}\right) \prod\limits_{\substack{p \text{ prime} \\ p\mid \frac{N}{\ell}}} \frac{1-\chi_t(p)p^{-1}}{1-p^{-2}} \sum\limits_{\substack{a \mid m \\ \gcd(a,N) = 1}} \mu(a) \chi_t(a) \sigma_{\ell,N,1}\left(\frac{m}{a}\right) & \begin{array}{@{}l} \text{if } -n=tm^2, \\ t \text{ fundamental}, \end{array} \\
0 & \text{else}.
\end{cases}
\end{align}
Note that this includes the definition of the coefficients $H_{1,N}(n)$ and that the $H_{1,1}(n)$ are the classical Hurwitz class numbers.

\begin{rmk}
Here, we included the summation condition that $\gcd(a,N)=1$ in Pei's and Wang's sums running over $a \mid m$ inside the definition of $H_{\ell,N}(n)$ for all $\ell \mid N$. This is justified, since these sums arise in their paper from the same sums as in our Lemma \ref{lem:TdivisorSum}. In addition to that, \eqref{eq:principalcharactertwist} below requires to twist these sums by $\left(\frac{N}{a}\right)^2$, which is equivalent to imposing the extra summation condition $\gcd(a,N)=1$, otherwise \eqref{eq:twistimpact} would not be correct.
\end{rmk}

\section{The plus space Kloosterman zeta function in higher levels} \label{sec:Kloostermanzeta}

If $N=1$, the plus space Kloosterman zeta function is evaluated in \cite{dit11annals}*{Lemma 4}, \cite{andu}*{Proposition 5.7}, \cite{ibsa}*{Section 2} or \cite{wong}*{Propositions 1 and 2} (after applying \cite{koh85}*{Proposition 5}) for example.

\subsection{General results} \label{subsec:Kloostermanzetageneral}

We begin by utilizing the following result by Shimura as well as two lemmas by Sturm\footnote{Note that there is a typo in Sturm's citation of Shimura's result, it should be $a^{\frac{1-L}{2}-\sfr}b^{2-L-2\sfr}$ in the definition of the coefficients $\beta$.}.
\begin{lemma}[\protect{\cite{shimura75procl}*{Proposition 1, (3.9)}, \cite{sturm}*{pp.\ $226$, $227$}}] \label{lem:shimurasturm}
Let $\Nc \in \N$ such that $\Nc \equiv 0 \pmod*{4}$, $\lambda = \frac{L}{2}$, $L \in 2\N-1$. Let $\omega$ be a Dirichlet character modulo $\Nc$ such that $\omega(-1) = 1$, and put 
\begin{itemize}
\item $\omega_{tm^2} \coloneqq \left(\frac{-1}{\cdot}\right)^{\frac{1-L}{2}}\left(\frac{t\Nc}{\cdot}\right)\omega$, whenever $n=tm^2$ with $t\neq 0$ square-free, and $\gcd(\cdot, t\Nc)=1$,
\item $\omega^* = \left(\frac{\Nc}{\cdot}\right)\omega$.
\end{itemize}
For every $c \in \N$, write $\Nc c=Md$, where $\gcd(d,M)=1$ and $M$ divides all sufficiently high powers of $\Nc$. Denote this last condition by $M \mid \Nc^{\infty}$.
\begin{enumerate}[label={\rm (\roman*)}]
\item Define\footnote{This fixes a typo in \cite{sturm}*{(4)}, the correct definition can be found in \cite{shimura75procl}*{(3.4)}.}
\begin{align*}
b(n,\sfr,\omega) \coloneqq \sum_{d > 0} \frac{\left(\frac{-\Nc}{d}\right)\varepsilon_d^L \omega(d)}{d^{\sfr+\lambda}} \sum_{r=1}^{d} \left(\frac{r}{d}\right) e^{2\pi i\frac{nr}{d}}.
\end{align*}
Suppose that $n \neq 0$. Then
\begin{align*}
b(n,\sfr,\omega) = \frac{L_{\Nc}\left(\sfr-\frac{1-L}{2},\omega_n\right)}{L_{\Nc}(2\sfr-(1-L),\omega^2)}\sum_{\substack{a,b \in \N \\ \gcd(a,\Nc)=1 \\ \gcd(b,\Nc)=1 \\ (ab)^2 \mid n }} \mu(a)\omega_n(a)\omega(b)^2 a^{\frac{1-L}{2}-s}b^{2-L-2\sfr}.
\end{align*}
\item We have
\begin{multline*}
\hspace*{\leftmargini}
\sum_{r =1}^{\Nc c} \omega(r) \varepsilon_r^{L}
\left(\frac{\Nc c}{r}\right) e^{2\pi i \frac{nr}{\Nc c}} \\
= \left(\left(\frac{-\Nc}{d}\right)\varepsilon_d^L \omega^*(d) \sum_{r=1}^{d} \left(\frac{r}{d}\right) e^{2\pi i\frac{nr}{d}}\right)\left(\sum_{r=1}^M \omega(r)\varepsilon_r^{L} \left(\frac{M}{r}\right)e^{2\pi i \frac{nr}{M}}\right).
\end{multline*}
\item Define
\begin{align*}
c(n,\sfr,\omega) \coloneqq \sum_{\Nc \mid M \mid \Nc^{\infty}} \frac{1}{M^{\lambda+\sfr}} \sum_{r=1}^M \omega(r)\varepsilon_r^{L} \left(\frac{M}{r}\right)e^{2\pi i \frac{nr}{M}}.
\end{align*}
Then,
\begin{align*}
\sum_{c \geq 1} \frac{1}{(\Nc c)^{\sfr+\lambda}}\sum_{r=1}^{\Nc c} \omega(r)\varepsilon_r^{L} \left(\frac{\Nc c}{r}\right)e^{2\pi i \frac{nr}{\Nc c}} = b(n,\sfr,\omega^*)c(n,\sfr,\omega),
\end{align*}
and $c(n,\sfr,\omega)$ is a finite Dirichlet series if $n \neq 0$.
\item If $n=0$ then
\begin{align*}
b(0,\sfr,\omega) = \frac{L_{\Nc}\left(2\sfr+L-2, \omega^2\right)}{L_{\Nc}\left(2\sfr-(1-L), \omega^2\right)}.
\end{align*}
\end{enumerate}
\end{lemma}

We cite the following results by Pei \cite{peiI} and Maass \cite{maass1937}.
\begin{lemma}[\protect{\cite{peiI}*{Lemma 2.1}, \cite{maass1937}*{p.\ $140$}}] \label{lem:maasseval1}
Let
\begin{align*}
a\left(p^j,n\right) \coloneqq \begin{cases}
\sum\limits_{r=1}^{2^j} \left(\frac{2^j}{r}\right)\varepsilon_{r} e^{2\pi i\frac{nr}{2^j}} & \text{if } p=2, \\[10pt]
\varepsilon_{p^j}^{-1} \sum\limits_{r=1}^{p^j} \left(\frac{r}{p^j}\right)e^{2\pi i\frac{nr}{p^j}} & \text{if } p > 2,
\end{cases}
\end{align*}
where $j \in \N$. 
\begin{enumerate}[label={\rm (\roman*)}]
\item Then, we have
\begin{align*}
\sum_{\Nc \mid M \mid \Nc^{\infty}} \frac{1}{M^{\sfr+\lambda}} \sum_{r=1}^M \varepsilon_r \left(\frac{M}{r}\right)e^{2\pi i \frac{nr}{M}} = \prod_{\substack{p \text{ prime} \\ p \mid \Nc}} \sum_{j\geq\nu_p(\Nc)} \frac{a\left(p^j,n\right) }{p^{j\left(\sfr+\lambda\right)}}.
\end{align*}
\item Let
\begin{align*}
\delta(x) \coloneqq \begin{cases}
1 & \text{if } x \text{ is an integer}, \\
0 & \text{otherwise},
\end{cases}
\end{align*}
and $\phi$ be Euler's totient function. Then, we have
\begin{align*}
a\left(2^j,n\right) = e^{\frac{\pi i}{4}} \cdot \begin{cases}
0, & \text{if } 2 \mid j, 2^{j-2} \nmid n, \\
2^{j-\frac{3}{2}}e^{\pi i \frac{m}{2}}, & \text{if } 2 \mid j, 2^{j-2} \mid n, m = \frac{n}{2^{j-2}}, 2 \mid m, \\
2^{j-\frac{3}{2}}e^{\pi i \frac{m-1}{2}}, & \text{if } 2 \mid j, 2^{j-2} \mid n, m = \frac{n}{2^{j-2}}, 2 \nmid m, \\
2^{j-1}\delta\left(\frac{u-1}{4}\right)e^{\pi i \frac{u-1}{4}}, & \text{if } 2 \nmid j, n=2^{j-3}u, 2\nmid u, \\
0, & \text{if } 2 \nmid j, n=2^{j-3}u, 2\mid u.
\end{cases}
\end{align*}
If $p > 2$, we have
\begin{align*}
a\left(p^j,n\right) = \begin{cases}
0, & \text{if } p^{j-1} \nmid n,\\
p^{j-\frac{1}{2}}\left(\frac{\frac{n}{p^{j-1}}}{p}\right), & \text{if } p^{j-1} \mid n, p^j \nmid n, 2 \nmid j, \\
0, & \text{if }  p^{j-1} \mid n, p^j \mid n, 2 \nmid j, \\
-p^{j-1}, & \text{if } p^{j-1} \mid n, p^j \nmid n, 2 \mid j, \\
\phi\left(p^j\right), & \text{if } p^{j-1} \mid n,  p^j \mid n, 2 \mid j.
\end{cases}
\end{align*}
\end{enumerate}
\end{lemma}

\begin{proof}
\
\begin{enumerate}[label={\rm (\roman*)}]
\item This follows by quadratic reciprocity.
\item In \cite{maass1937}*{p.\ $140$}, Maass evaluates closely related numbers $a_c(n)$, which he defines in \cite{maass1937}*{(7), (10)}. He notes that if $\gcd(c_1,c_2)=1$ then $a_{c_1c_2}(n) = a_{c_1}(n) a_{c_2}(n)$. Hence, it suffices to evaluate on prime powers. His numbers $a_{p^j}(n)$ are related to the numbers $a(p^j,n)$ from the Lemma by the identities
\begin{align} \label{eq:multipliersumrewrite}
e^{-\frac{\pi i}{4}}a\left(2^{j+1},n\right) = a_{2^j}(n), \qquad a\left(p^j,n\right) = a_{p^j}(n)
\end{align}
for $p > 2$. The claim follows by the special case that $L=1$ in Maass' result. \qedhere
\end{enumerate}
\end{proof}

Furthermore, we rewrite the divisor function type sums in Lemma \ref{lem:shimurasturm}. Define
\begin{align*}
T_{\Nc,\sfr}^{\chi}(n) \coloneqq \sum_{\substack{0 < d \mid n \\ \gcd(d,\Nc)=1}} \mu(d)\chi(d)d^{\sfr-1}\sigma_{\Nc,2\sfr-1}\left(\frac{n}{d}\right).
\end{align*}

In terms of this notation, we have the following technical result.
\begin{lemma} \label{lem:TdivisorSum}
Let $\chi$ be a Dirichlet character and $n = tm^2$ with $t$ square-free. Then, we have
\begin{align*}
\sum_{\substack{a,b \in \N \\ \gcd(a,\Nc)=1 \\ \gcd(b,\Nc)=1 \\ (ab)^2 \mid n}} \mu(a) \chi(a) a^{-\sfr}b^{1-2\sfr} = T_{\Nc,1-\sfr}^{\chi}(m).
\end{align*}
\end{lemma}

\begin{proof}
We substitute $c_1\coloneqq ab$, $c_2 \coloneqq a$ first and write $c_1 = c_2c_3$ subsequently, getting
\begin{align*}
\sum_{\substack{a,b \in \N \\ \gcd(a,\Nc)=1 \\ \gcd(b,\Nc)=1 \\ (ab)^2 \mid n}} \mu(a)\chi(a) a^{-\sfr}b^{1-2\sfr} &= \sum_{\substack{0 < c_1 \mid m \\ \gcd(c_1,\Nc)=1}} c_1^{1-2\sfr} \sum_{\substack{0 < c_2 \mid c_1 \\ \gcd(c_2,\Nc)=1}} \mu(c_2)\chi(c_2) c_2^{\sfr-1} \\
&= \sum_{\substack{0 < c_2 \mid m \\ \gcd(c_2,\Nc)=1}} \mu(c_2)\chi(c_2) c_2^{\sfr-1} \sum_{\substack{c_2 \mid c_1 \mid m \\ \gcd(c_1,\Nc)=1}} c_1^{1-2\sfr} = T_{\Nc,1-\sfr}^{\chi}(m). \qedhere
\end{align*}
\end{proof}

\subsection{Application to our plus space Kloosterman zeta function in weight \texorpdfstring{$1/2$}{1/2} and at \texorpdfstring{$s=3/4$}{s=3/4}}

In weight $\frac{1}{2}$, we are interested in the specialization at the spectral point $\frac{3}{4}$. Note that we evaluated the plus space Kloosterman zeta function at $\sfr+\frac{1}{2}$ so far, while our Fourier expansion from Proposition \ref{prop:Eisensteinprojected} evaluates it at $2s$. Hence, the connection between Proposition \ref{prop:Eisensteinprojected} and the results from Subsection \ref{subsec:Kloostermanzetageneral} is given by
\begin{align*}
\sfr = 2s-\frac{1}{2}, \qquad s = \frac{\sfr}{2} + \frac{1}{4},
\end{align*}
which imply that we are interested in the case $\sfr = 1$. We begin by employing Lemmas \ref{lem:shimurasturm}, \ref{lem:maasseval1} to $\Ks_{k,4N}^+$ from \eqref{eq:Kloostermanzetadefinition}.
\begin{prop} \label{prop:Kloostermanzetageneral}
Write $n = tm^2$ with $t$ square-free.
\begin{enumerate}[label={\rm (\roman*)}]
\item We have
\begin{align*}
\hspace*{\leftmargini} \Ks_{\frac{1}{2},4N}^+\left(0,0;\sfr+\frac{1}{2}\right)
= \frac{L_{4N}\left(2\sfr-1,\id\right)}{L_{4N}\left(2\sfr,\id\right)} \left(\sum_{j\geq2} \frac{a\left(2^j,0\right) }{2^{j\left(\sfr+\frac{1}{2}\right)}} + \frac{a\left(4,0\right)}{2^{2\left(\sfr+\frac{1}{2}\right)}}\right) \prod_{\substack{p \text{ prime} \\ p \mid N}} \sum_{j\geq1} \frac{a\left(p^j,0\right)}{p^{j\left(\sfr+\frac{1}{2}\right)}}.
\end{align*}
\item If $n \neq 0$, then
\begin{multline*}
\hspace*{\leftmargini} \Ks_{\frac{1}{2},4N}^+\left(0,n;\sfr+\frac{1}{2}\right) \\
= \frac{L_{4N} (\sfr, \chi_t)}{L_{4N} (2\sfr, \id)} T_{4N,1-\sfr}^{\chi_t}(m) \left(  \sum_{j \geq 2} \frac{ a(2^j,n)}{2^{j(\sfr + \frac{1}{2})}} + \frac{ a(4,n)}{2^{2(\sfr + \frac{1}{2})}} \right) \prod_{\substack{p \text{ prime} \\ p \mid N}} \sum_{j \geq 1} \frac{a(p^j,n)}{p^{j(\sfr + \frac{1}{2})}} .
\end{multline*}
\end{enumerate}
\end{prop}

\begin{proof}
\
\begin{enumerate}[label={\rm (\roman*)}]
\item Choosing $\Nc = 4N$, $\omega=\chi_{4N}^2$ to be the principal character of modulus $4N$ and $L=1$, we compute that
\begin{align} 
\omega^* &= \chi_{4N}\omega = \chi_{4N}^3 = \chi_{4N}, \qquad \left(\omega^*\right)_n = (\chi_{4N})_n = \chi_{4N}^2\chi_t, \label{eq:principalcharactertwist}
\end{align}
and note that
\begin{align} \label{eq:twistimpact}
\begin{split}
\hspace*{\leftmargini} & L_{4N}\left(\sfr,\chi_{4N}^2\chi_t\right) = \sum_{\substack{n \geq 1 \\ \gcd(n,4N)=1}} \frac{\chi_{4N}^2(n)\chi_t(n)}{n^s} = L_{4N}\left(\sfr,\chi_t\right), \\
&T_{4N,1-\sfr}^{\chi_{4N}^2\chi_t}(m) = \sum_{\substack{0 < d \mid m \\ \gcd(d,4N)=1}} \mu(d)\chi_{4N}^2(d)\chi_t(d)d^{\sfr-1}\sigma_{4N,2\sfr-1}\left(\frac{m}{d}\right) = T_{4N,1-\sfr}^{\chi_t}(m).
\end{split}
\end{align}
By virtue of Lemmas \ref{lem:shimurasturm} and \ref{lem:TdivisorSum} we infer that
\begin{multline*} 
\hspace*{\leftmargini} \sum_{c > 0} \frac{1}{(4Nc)^{\sfr+\frac{1}{2}}} \sum_{\substack{r \pmod*{4Nc} \\ \gcd(4Nc,r)=1}} \left(\frac{4Nc}{r}\right)\varepsilon_r e^{2\pi i \frac{nr}{4Nc}} \\
= \frac{L_{4N}\left(\sfr,\chi_t\right)}{L_{4N}(2\sfr,\id)} T_{4N,1-\sfr}^{\chi_t}(m) \sum_{(4N) \mid M \mid (4N)^{\infty}} \frac{1}{M^{\sfr+\frac{1}{2}}} \sum_{r=1}^M \varepsilon_r \left(\frac{M}{r}\right)e^{2\pi i \frac{nr}{M}}.
\end{multline*}
The sum running over odd $c > 0$ is evaluated by writing
\begin{align*}
\sum_{\substack{c > 0 \\ c \text{ odd}}} = \sum_{c > 0} - \sum_{\substack{c > 0 \\ c \text{ even}}},
\end{align*}
and mapping $c \mapsto 2c$ in the latter sum. The former sum is evaluated as above, while the latter sum can be evaluated by choosing $\Nc = 8N$, $\omega=\chi_{8N}^2$ and $L=1$ in Lemma \ref{lem:shimurasturm}. Since $N$ is odd, we obtain
\begin{multline*}
\hspace*{\leftmargini} \sum_{\substack{c > 0 \\ c \text{ odd}}} \frac{1}{(4Nc)^{\sfr+\frac{1}{2}}} \sum_{\substack{r \pmod*{4Nc} \\ \gcd(4Nc,r)=1}} \left(\frac{4Nc}{r}\right)\varepsilon_r e^{2\pi i \frac{nr}{4Nc}} \\
= \frac{L_{4N}\left(\sfr,\chi_t\right)}{L_{4N}(2\sfr,\id)} T_{4N,1-\sfr}^{\chi_t}(m) \left(\sum_{(4N) \mid M \mid (4N)^{\infty}} - \sum_{(8N) \mid M \mid (8N)^{\infty}}\right)
\frac{1}{M^{\sfr+\frac{1}{2}}} \sum_{r=1}^M \varepsilon_r \left(\frac{M}{r}\right)e^{2\pi i \frac{nr}{M}} \\
= \frac{L_{4N}\left(\sfr,\chi_t\right)}{L_{4N}(2\sfr,\id)} T_{4N,1-\sfr}^{\chi_t}(m) \sum_{\substack{(4N) \mid M \mid (4N)^{\infty} \\ \nu_2(M) = 2}} \frac{1}{M^{\sfr+\frac{1}{2}}} \sum_{r=1}^M \varepsilon_r \left(\frac{M}{r}\right)e^{2\pi i \frac{nr}{M}}.
\end{multline*}
Thus, the first assertion follows by Lemma \ref{lem:maasseval1}.

\item The second assertion is proven along the same lines as the first one. Additionally, Lemma \ref{lem:TdivisorSum} is used here. \qedhere
\end{enumerate}
\end{proof}

Next, we specialize Lemma \ref{lem:maasseval1} to $k=\frac{1}{2}$ and $\sfr = 1$. Consequently, we define
\begin{align*}
A(2,n) \coloneqq \sum_{j\geq2} \frac{a\left(2^{j},n\right) }{2^{\frac{3}{2}j}}, \qquad A(p,n) \coloneqq \sum_{j\geq1} \frac{a\left(p^{j},n\right) }{p^{\frac{3}{2}j}},
\end{align*}
and note that
\begin{align} \label{eq:localoddcontribution}
a(4,n) = \begin{cases}
1+i & \text{if } n \equiv 0,1 \pmod*{4}, \\
-1-i & \text{if } n \equiv 2,3 \pmod*{4}.
\end{cases}
\end{align}

Let $\nu_p(m) \in \N_0 \cup \{\infty\}$ be the $p$-adic valuation of $m \in \Z$.
\begin{prop} \label{prop:localfactors}
Let $n \neq 0$, $\sfr=1$, and $p$ be an odd prime.
\begin{enumerate}[label={\rm (\roman*)}]
\item We have
\begin{align*}
A(2,n) = \frac{1+i}{4} \cdot \begin{cases} 1-3\cdot 2^{-\frac{\nu_2(n)+1}{2}} & \text{if } 2 \nmid \nu_2(n), \\
 1-3\cdot 2^{-\frac{\nu_2(n)}{2}-1} & \text{if } 2 \mid \nu_2(n), \frac{n}{2^{\nu_2(n)}} \equiv 3 \pmod*{4}, \\
1 & \text{if } 2 \mid \nu_2(n), \frac{n}{2^{\nu_2(n)}} \equiv 1 \pmod*{8}, \\
1-2^{-\frac{\nu_2(n)}{2}} & \text{if } 2 \mid \nu_2(n), \frac{n}{2^{\nu_2(n)}} \equiv 5 \pmod*{8},
\end{cases}
\end{align*}
and
\begin{align*}
A(p,n) = \begin{cases}
\frac{1}{p}-(p+1)p^{-\frac{\nu_p(n)+3}{2}} & \text{if } 2 \nmid \nu_p(n), \\
\frac{1}{p} & \text{if } 2 \mid \nu_p(n), \left(\frac{\frac{n}{p^{\nu_p(n)}}}{p}\right) = 1, \\
\frac{1}{p} - 2p^{-\frac{\nu_p(n)}{2}-1} & \text{if } 2 \mid \nu_p(n), \left(\frac{\frac{n}{p^{\nu_p(n)}}}{p}\right) = -1.
\end{cases}
\end{align*}

\item In particular, we have 
\begin{align*} 
A\left(2,m^2\right) = \frac{1+i}{4}, \qquad A\left(p,m^2\right) = \frac{1}{p}
\end{align*}
for every $m \in \Z$.
\end{enumerate}
\end{prop}

\begin{proof}
\
\begin{enumerate}[label={\rm (\roman*)}]
\item To prove the first assertion, we cite the following calculations by Maass \cite{maass1937}*{pp.\ 141, 142} to evaluate the sums from Lemma \ref{lem:maasseval1}. To this end, let $\Dc(n)$ be the discriminant of $\Q\left(\sqrt{n}\right)$ and recall \eqref{eq:multipliersumrewrite}. 

If $ p = 2$ then
\begin{multline*}
\hspace*{\leftmargini} \sum_{j\geq0} \frac{a\left(2^{j+1},n\right) }{2^{j\left(\sfr+\frac{1}{2}\right)}} = e^{\frac{\pi i}{4}} \times \begin{cases}
1 + \left(\frac{\Dc(n)}{2}\right)2^{-\sfr} & \text{if } 2 \nmid \Dc(n), \\
1 - 2^{-2\sfr} & \text{if } 2 \mid \Dc(n),
\end{cases} 
\\
\times \begin{cases}
\frac{1}{1+2^{-\sfr}}\Bigg(\sum\limits_{j=0}^{\frac{\nu_2(n)-1}{2}}2^{j(1-2\sfr)} + 2^{1-\sfr}\sum\limits_{j=0}^{\frac{\nu_2(n)-3}{2}}2^{j(1-2\sfr)} \Bigg) & \text{if } 2 \nmid \nu_2(n), \\
-\frac{2^\sfr}{1+2^{-\sfr}} + 2^\sfr\sum\limits_{j=0}^{\frac{\nu_2(n)}{2}} 2^{j(1-2\sfr)} & \text{if } 2 \mid \nu_2(n), \frac{n}{2^{\nu_2(n)}} \equiv 3 \pmod*{4}, \\
\left(1+2^{1-\sfr}\right)\sum\limits_{j=0}^{\frac{\nu_2(n)}{2}} 2^{j(1-2\sfr)} & \text{if } 2 \mid \nu_2(n), \frac{n}{2^{\nu_2(n)}} \equiv 5 \pmod*{8}, \\
\frac{2}{1+2^{-\sfr}} + \left(2^{1-\sfr}-1\right)\sum\limits_{j=0}^{\frac{\nu_2(n)}{2}} 2^{j(1-2\sfr)} & \text{if } 2 \mid \nu_2(n), \frac{n}{2^{\nu_2(n)}} \equiv 1 \pmod*{8}.
\end{cases}
\end{multline*}
If $ p > 2$ then
\begin{multline*}
\hspace*{\leftmargini} \sum_{j\geq0} \frac{a\left(p^j,n\right) }{p^{j\left(\sfr+\frac{1}{2}\right)}} = \\
\hspace*{\leftmargini} \begin{cases} \left(1-p^{-2\sfr}\right) \sum\limits_{j=0}^{\frac{\nu_p(n)-1}{2}} p^{j(1-2\sfr)} & \text{if } p \mid \Dc(n), \\
\left(1+\left(\frac{\Dc(n)}{p}\right)p^{-\sfr}\right) \Bigg(p^{(1-2\sfr)\frac{\nu_p(n)}{2}} + \left(1-\left(\frac{\Dc(n)}{p}\right)p^{-\sfr}\right)\sum\limits_{j=0}^{\frac{\nu_p(n)}{2}-1} p^{j(1-2\sfr)}\Bigg) & \text{if } p \nmid \Dc(n).
\end{cases}
\end{multline*}
After some calculations, we obtain
\begin{align} \label{eq:localfactor2s}
\begin{split}
\hspace*{\leftmargini}& \sum_{j\geq2} \frac{a\left(2^{j},n\right) }{2^{j\left(\sfr+\frac{1}{2}\right)}}
= \frac{1+i}{2^{\sfr+1}} \\
&\times \begin{cases} \frac{2^{-\nu_2(n)\sfr}\left(2^{(\nu_2(n)+1)\sfr}-2^{\frac{\nu_2(n)+1}{2}}\left(2^{2\sfr}-1\right)\right)}{2^{2\sfr}-2} & \text{if } 2 \nmid \nu_2(n), \\
\frac{2^{-(\nu_2(n)+1)\sfr}\left(2^{(\nu_2(n)+2)\sfr} - 2^{\frac{\nu_2(n)}{2}+2\sfr+1} + 2^{\frac{\nu_2(n)}{2}+1}\right)}{2^{2\sfr}-2} & \text{if } 2 \mid \nu_2(n), \frac{n}{2^{\nu_2(n)}} \equiv 3 \pmod*{4}, \\
\frac{2^{-(\nu_2(n)+2)\sfr}\left(2^{\frac{\nu_2(n)}{2}+1}\left(2^\sfr-2\right)\left(2^\sfr+1\right)+2^{(\nu_2(n)+3)\sfr}\right)}{2^{2\sfr}-2} & \text{if } 2 \mid \nu_2(n), \frac{n}{2^{\nu_2(n)}} \equiv 1 \pmod*{8}, \\
\frac{2^{-(\nu_2(n)+2)\sfr}\left(-2^{\frac{\nu_2(n)}{2}+1}\left(2^\sfr+2^{2\sfr}-2\right)+2^{(\nu_2(n)+3)\sfr}\right)}{2^{2\sfr}-2} & \text{if } 2 \mid \nu_2(n), \frac{n}{2^{\nu_2(n)}} \equiv 5 \pmod*{8},
\end{cases}
\end{split}
\end{align}
and
\begin{align} \label{eq:localfactorps}
\begin{split}
\hspace*{\leftmargini}& 1 + \sum_{j\geq1} \frac{a\left(p^{j},n\right) }{p^{j\left(\sfr+\frac{1}{2}\right)}} = \\
& \begin{cases}
\left(1-p^{-2\sfr}\right)\frac{p^{\sfr(1-\nu_p(n))+\frac{\nu_p(n)+1}{2}}-p^{2\sfr}}{p-p^{2\sfr}} & \text{if } 2 \nmid \nu_p(n), \\
\frac{-p^{\frac{\nu_p(n)}{2}-\nu_p(n)\sfr+2\sfr}+p^{\frac{\nu_p(n)}{2}-\nu_p(n)\sfr}\left(p^{2\sfr} - p^{1-\sfr} + p^\sfr - p + 1\right) + p^{2\sfr} - 1}{p^{2\sfr} - p} & \text{if } 2 \mid \nu_p(n), \left(\frac{\frac{n}{p^{\nu_p(n)}}}{p}\right) = 1, \\
\frac{-p^{\frac{\nu_p(n)}{2}-\nu_p(n)\sfr + 2\sfr} + p^{\frac{\nu_p(n)}{2}-\nu_p(n)\sfr} \left(p^{2\sfr} + p^{1-\sfr} - p^\sfr - p + 1\right)+ p^{2\sfr} - 1}{p^{2\sfr} - p} & \text{if } 2 \mid \nu_p(n), \left(\frac{\frac{n}{p^{\nu_p(n)}}}{p}\right) = -1
\end{cases}
\end{split}
\end{align}
for $p > 2$. The first assertion follows by specializing this to $\sfr = 1$. Alternatively, we combine \eqref{eq:Kloostermanfunctionalequation} (or Lemma \ref{lem:peiwangconnection} (i) below) with Pei \cite{peiI}*{(2.17), (2.18)}. 

\item If $m \neq 0$, the claim is a direct consequence of the first assertion, since $\nu_p\left(m^2\right)$ is even and both $3$ resp.\ $5$ are quadratic non-residues modulo $4$ resp.\ $8$. If $m=0$, we let $\nu_p(n) \to \infty$ in each arithmetic progression in the first assertion. Alternatively, we combine \eqref{eq:Kloostermanfunctionalequation} (or Lemma \ref{lem:peiwangconnection} (i) below) with Pei \cite{peiI}*{(2.19)}. \qedhere
\end{enumerate}
\end{proof}

Lastly, we evaluate the result of Lemma \ref{lem:TdivisorSum} at $\sfr=1$.
\begin{lemma} \label{lem:TdivisorSpecial}
Let $p$ be a prime and $\nu_p(m) \geq 1$. Then, we have
\begin{align*}
T_{4N,0}^{\chi}\left(p^{\nu_p(m)}\right) = 
\begin{cases}
1 + (1-\chi(p)) \sum\limits_{j=1}^{\nu_p(m)} p^{-j} & \text{if } \gcd(p, 4N) = 1, \\
1 & \text{if } \gcd(p, 4N) = p.
\end{cases}
\end{align*}
In particular, we have $T_{4N,0}^{\chi_{4N}^2}(m)=1$.
\end{lemma}

\begin{proof}
Note that $T_{4N,\sfr}^{\chi}(m)$ is a multiplicative function of $m$, because the Dirichlet convolution of multiplicative functions is multiplicative, and the $\gcd$ is multiplicative with one of its inputs fixed. The first claim follows by evaluating the functions involved. The second claim is an immediate consequence of the first one.
\end{proof}

\subsection{The residue of \texorpdfstring{$\Ks_{\frac{1}{2},4N}^+$}{the plus space Kloosterman zeta function} at \texorpdfstring{$s=\frac{3}{4}$}{s = 3/4}} \label{subsec:Klostermanzetaresidues}

We compute the residue of $\Ks_{\frac{1}{2},4N}^+\left(0,n;2s\right)$ at $s=\frac{3}{4}$. To this end, we abbreviate
\begin{align} \label{eq:resiudemultiple}
\rf(N) \coloneqq \frac{1}{2\pi} \prod_{\substack{p \text{ prime} \\ p \mid N}} \frac{1}{p+1}.
\end{align}

\begin{prop} \label{prop:Kloostermanzetaresidue}
We have
\begin{align*}
\lim_{s \to \frac{3}{4}^+} \left(s-\frac{3}{4}\right) \Ks_{\frac{1}{2},4N}^+\left(0,n;2s\right) = \begin{cases}
\frac{3}{2\pi}(1+i) \rf(N) & \text{if } n=m^2 \geq 1, \\
\frac{3}{4\pi}(1+i) \rf(N) & \text{if } n=0, \\
0 & \text{otherwise}.
\end{cases}
\end{align*}
\end{prop}

\begin{proof}
If $n = 0$, we utilize Propositions \ref{prop:Kloostermanzetageneral} and obtain
\begin{multline*}
\lim_{s \to \frac{3}{4}^+} \left( s - \frac{3}{4} \right) \Ks_{\frac{1}{2},4N}^+(0,0;2s) = \lim_{s \to \frac{3}{4}^+} \left(s - \frac{3}{4} \right)  \frac{L_{4N}\left(2 \left(2s-\frac{1}{2}\right)-1,\id\right)}{L_{4N}\left(2 \left(2s-\frac{1}{2}\right),\id\right)} \\
\times \left(\sum_{j\geq2} \frac{a\left(2^j,0\right) }{2^{j\left(2s-\frac{1}{2}+\frac{1}{2}\right)}} + \frac{a\left(4,0\right)}{2^{2\left(2s-\frac{1}{2}+\frac{1}{2}\right)}}\right)  \prod_{\substack{p \text{ prime} \\ p \mid N}} \sum_{j\geq1} \frac{a\left(p^j,0\right)}{p^{j\left(2s-\frac{1}{2}+\frac{1}{2}\right)}}.
\end{multline*}
According to Proposition \ref{prop:localfactors} and \eqref{eq:localoddcontribution}, we infer that
\begin{align*}
&\lim_{s \to \frac{3}{4}^+} \left( s - \frac{3}{4} \right) \Ks_{\frac{1}{2},4N}^+(0,0;2s) \\
&= \frac{1}{L_{4N}\left(2,\id\right)} \cdot \frac{1}{N} \cdot \left( \frac{1 + i}{4} + \frac{(1+i)}{8} \right) \lim_{s \to \frac{3}{4}^+} \frac{1}{4} (4s - 2 -1) L_{4N}\left(4s-2,\id\right) \\
&= \frac{6}{\pi^2} \cdot \frac{1}{N} \cdot \frac{3(1+i)}{8} \cdot \prod_{\substack{p \text{ prime} \\ p \mid 4N}} \frac{p}{p+1} \cdot \lim_{s \to \frac{3}{4}^+} \frac{1}{4} (4s - 2 -1) \zeta(4s -2)
= \frac{3(1+i)}{4 \pi} \rf(N).
\end{align*}

For $n = m^2 \geq 1$, the computation is similar, since $T_{4N,0}^{\chi_{4N}^2}(m)=1$ by Lemma \ref{lem:TdivisorSpecial} (ii) and the evaluation of the local factors is the same by comparing expressions in Propositions \ref{prop:Kloostermanzetageneral}, \ref{prop:localfactors}. Hence, the only difference to the previous computation is that we have a factor of
\begin{align*}
\lim_{s \to \frac{3}{4}^+}  \left(s - \frac{3}{4}\right) \zeta \left(2s-\frac{1}{2}\right) = \frac{1}{2}
\end{align*}
instead of $\frac{1}{4}$. This proves the claimed evaluations.
\end{proof}

\subsection{The constant term in the Laurent expansion of \texorpdfstring{$\Ks_{\frac{1}{2},4N}^+$}{the plus space Kloosterman zeta function} at \texorpdfstring{$s=\frac{3}{4}$}{s = 3/4}} \label{subsec:analyticcontinuation}
Let $\gamma$ be the Euler--Mascheroni constant. We compute the Fourier expansion of $\Gc$ in terms of
\begin{align} \label{eq:cfdef}
\cf\left(n\right) \coloneqq \frac{\partial}{\partial s} \left[ \left(s-\frac{3}{4}\right) \Ks_{\frac{1}{2},4N}^+(0,n;2s) \right] \Bigg\vert_{s=\frac{3}{4}}.
\end{align}
We observe that the function $\cf$ is well-defined for every $n \in \Z$. If $n \geq 0$ is a square, we recall that
$
\frac{\partial}{\partial s} \left[ \left(s-1\right) \zeta(s) \right] \Big\vert_{s=1} = \gamma,
$
see \cite{nist}*{item 25.2.4} for example. This implies that $\cf(n)$ exists by virtue of Proposition \ref{prop:Kloostermanzetageneral} and \eqref{eq:L_Ndef}. If $n \geq 0$ is not a square or if $n < 0$ (with $n \equiv 0,1 \pmod*{4}$ in both cases) then $\Ks_{\frac{1}{2},4N}^+(0,n;2s)$ is regular at $s=\frac{3}{4}$, because the involved Dirichlet $L$-functions are. In other words, we have
\begin{align} \label{eq:KloostermanZetaResidueRegular}
\cf(n) = \Ks_{\frac{1}{2},4N}^+\left(0,n;\frac{3}{2}\right), \text{ whenever } n < 0 \text{ or } 0 \leq n \neq \square.
\end{align}
An explicit formula for $\cf(n)$ can be obtained by virtue of Lemma \ref{lem:TdivisorSum} and Propositions \ref{prop:Kloostermanzetageneral}, \ref{prop:localfactors}. If $n = m^2 \geq 0$ with $n \equiv 0,1 \pmod*{4}$, we utilize the same results and \eqref{eq:localoddcontribution}, which motivates to define
\begin{align*}
\ff\left(m^2,s\right) \coloneqq \left(\sum_{j\geq2} \frac{a\left(2^j,m^2\right)}{2^{js}} + \frac{1+i}{2^{2s}}\right) \prod_{\substack{p \text{ prime} \\ p \mid N}} \sum_{j\geq1} \frac{a\left(p^j,m^2\right)}{p^{js}} 
\cdot \begin{cases}
1 & \text{if } m = 0, \\
T_{4N,\frac{3}{2}-s}^{\id}(m) & \text{if } m \geq 1.
\end{cases}
\end{align*}
We refer to Proposition \ref{prop:localfactors} and \eqref{eq:localfactor2s}, \eqref{eq:localfactorps} for an evaluation of $\ff\left(m^2,s\right)$. In terms of $\ff\left(m^2,2s\right)$, we arrive at
\begin{multline} \label{eq:KloostermanZetaResidue0}
\cf(0) = \left(\frac{\partial}{\partial s} \left(s-\frac{3}{4}\right) \frac{L_{4N}(4s-2,\id)}{L_{4N}(4s-1, \id)}\right) \Bigg\vert_{s=\frac{3}{4}} \ff\left(0,\frac{3}{2}\right) \\
+ \left(\lim_{s \to \frac{3}{4}^+} \left(s-\frac{3}{4}\right) \frac{L_{4N}(4s-2,\id)}{L_{4N}(4s-1, \id)}\right) \frac{\partial}{\partial s} \ff(0,2s) \Bigg\vert_{s=\frac{3}{4}}
\end{multline}
and at
\begin{multline} \label{eq:KloostermanZetaResidueSquare}
\cf\left(m^2\right) = \left(\frac{\partial}{\partial s} \left(s-\frac{3}{4}\right) \frac{L_{4N}\left(2s-\frac{1}{2},\id\right)}{L_{4N}(4s-1, \id)}\right) \Bigg\vert_{s=\frac{3}{4}} \ff\left(m^2,\frac{3}{2}\right) \\
+ \left(\lim_{s \to \frac{3}{4}^+} \left(s-\frac{3}{4}\right) \frac{L_{4N}\left(2s-\frac{1}{2},\id\right)}{L_{4N}(4s-1, \id)}\right) \frac{\partial}{\partial s} \ff\left(m^2,2s\right) \Bigg\vert_{s=\frac{3}{4}}, \qquad m \geq 1.
\end{multline}
If $N=p$ is an odd prime, we evaluate both $\cf(0)$ and $\cf\left(m^2\right)$ in part II \cite{bemo2}*{Theorem 1.3}.

\section{Proof of Theorem \ref{thm:maingeneral} \texorpdfstring{$\mathrm{(i)}$}{(i)} and \texorpdfstring{$\mathrm{(ii)}$}{(ii)}} \label{sec:proofof1.3part1and2}

Since the Dirichlet $L$-functions involved have at most a simple pole at $\sfr=1$, we have the formal Laurent expansion
\begin{align} \label{eq:laurentformal}
\Fc_{\frac{1}{2},4N}^+(\tau,s) = \sum_{j \geq -1} f_j(\tau)\left(s-\frac{3}{4}\right)^j
\end{align}
about $s = \frac{3}{4}$. The functions $f_j$ are polyharmonic Maass forms. Roughly speaking, such forms are an inductive generalization of sesquiharmonic Maass forms to a vanishing condition under multiple applications of $\xi_k$. This statement can be found in \cite{anlagrho}*{middle of p.\ 43}, which in turn generalizes \cite{lagrho}*{Theorem 8.4} to half integral weights. See \cite{brika20}*{Lemma 3.2} as well. We recall the notation $\rf(N)$ from \eqref{eq:resiudemultiple}.

\subsection{The residue of the Eisenstein series at \texorpdfstring{$s=\frac{3}{4}$}{s = 3/4}}
Let
\begin{align} \label{eq:thetadef}
\theta(\tau) = \sum_{m \in \Z} q^{m^2} = 1 + 2\sum_{m \geq 1} q^{m^2} \in M_{\frac{1}{2}}(4),
\end{align}
be the standard Jacobi theta function. It can be shown that the residue of $\Fc_{\frac{1}{2},4N}^+$ satisfies
\begin{align*}
f_{-1}(\tau) \in M_{\frac{1}{2}}(4N),
\end{align*}
see \cite{duim}*{Section $4$} for example. The Serre--Stark \cite{sest} basis theorem asserts that a basis for the more general space $M_{\frac{1}{2}}(4M)$ with $M \in \N$ is given by (twisted) weight $\frac{1}{2}$ theta functions,  see \cite{thebook}*{Theorem 2.8} too. However, since our level equals $4N$ with $N$ odd and square-free, the space $M_{\frac{1}{2}}(4N)$ is in fact one-dimensional and generated by $\theta$. Alternatively, this can be justified by observing that the Dirichlet $L$-function from Proposition \ref{prop:Kloostermanzetageneral} has a simple pole at $\sfr=1$ if and only if $n=m^2 \geq 0$, since the involved character equals the principal character $\chi_{4N}^2$ of modulus $4N$ if and only if $t=1$ by \eqref{eq:principalcharactertwist}. The precise relation between the residue of $\Fc_{\frac{1}{2},4N}^+$ and $\theta$ is a direct consequence of Proposition \ref{prop:Kloostermanzetaresidue}.
\begin{prop} \label{prop:Eisensteinresidue}
We have
\begin{align*}
f_{-1}(\tau) = \lim_{s \to \frac{3}{4}^+} \left(s-\frac{3}{4}\right) \Fc_{\frac{1}{2},4N}^+(\tau,s) = \rf(N) \theta(\tau).
\end{align*}
\end{prop}

\begin{rmk}
References for the level $1$ case include \cite{ibsa}*{Proposition 2.3}, \cite{dit11annals}*{Lemma $3$, (5.4)}, \cite{alansa}*{(6.3)}. 
\end{rmk}

\begin{proof}[Proof of Proposition \ref{prop:Eisensteinresidue}]
We combine the results of Proposition \ref{prop:Eisensteinprojected} and Proposition \ref{prop:Kloostermanzetaresidue}. Recall \eqref{eq:whittakerspectral}. If $n = m^2 \geq 1$, this yields a contribution
\begin{align*}
\frac{2}{3}i^{-\frac{1}{2}}\pi^{\frac{3}{4}}(4 \pi)^{\frac{1}{4}} \frac{3}{2\pi}(1+i) \rf(N) \sum_{m \geq 1} q^{m^2} = 2 \rf(N) \sum_{m \geq 1} q^{m^2}.
\end{align*}
If $n = m^2 = 0$, this yields a contribution
\begin{align*}
\frac{2}{3} \frac{4^{1-\frac{3}{4}}\pi i^{-\frac{1}{2}}\Gamma\left(\frac{1}{2}\right)}{\Gamma\left(\frac{3}{4}+\frac{1}{4}\right)\Gamma\left(\frac{3}{4}-\frac{1}{4}\right)} \frac{3}{4\pi}(1+i) \rf(N) = \rf(N).
\end{align*}
Adding both contributions proves the claim.
\end{proof}

\subsection{Proof of Theorem \ref{thm:maingeneral} (i) and (ii)}
Recall that $\Gc(\tau) = f_0(\tau)$ in \eqref{eq:laurentformal}. We are in position to prove Theorem \ref{thm:maingeneral} (i) and (ii). 

\begin{proof}[Proof of Theorem \ref{thm:maingeneral} (i) and (ii)]
Modularity and the growth condition are clear from the Laurent expansion \eqref{eq:laurentformal}, compare the references stated after that equation. By virtue of \eqref{eq:Deltasplitting} and the fact that the $\xi_k$-operators intertwine with the slash operator as well as the plus space projection operator, we deduce that
\begin{align*}
\Delta_{\frac{1}{2}} \Fc_{\frac{1}{2},4N}^+(\tau,s) = \left(s-\frac{1}{4}\right)\left(\frac{3}{4}-s\right)\Fc_{\frac{1}{2},4N}^+(\tau,s),
\end{align*}
and thus
\begin{align} \label{eq:DeltaGresidue}
\Delta_{\frac{1}{2}} \Gc(\tau) = -\frac{1}{2} f_{-1}(\tau).
\end{align}
Since $f_{-1}$ is holomorphic, this implies that $\Gc$ is sesquiharmonic.

To verify the Fourier expansion, we inspect each term in Proposition \ref{prop:Eisensteinprojected} separately.
\begin{enumerate}
\item Clearly, we have
\begin{align*}
\frac{\partial}{\partial s} \frac{2}{3} \left(s-\frac{3}{4}\right)v^{s-\frac{1}{4}} \Big\vert_{s=\frac{3}{4}} = \frac{2}{3} v^{\frac{1}{2}}.
\end{align*}

\item We compute that
\begin{align*}
\frac{\partial}{\partial s} \frac{2}{3} \frac{4^{1-s}\pi i^{-\frac{1}{2}}\Gamma(2s-1)}{\Gamma\left(s+\frac{1}{4}\right)\Gamma\left(s-\frac{1}{4}\right)}  v^{1-s-\frac{1}{4}} \Big\vert_{s=\frac{3}{4}} = \frac{2}{3} (i-1)\pi\left(\log(16v)\right).
\end{align*}
Utilizing Proposition \ref{prop:Kloostermanzetaresidue}, we infer that
\begin{multline*}
\hspace*{\leftmargini} \left(\lim_{s \to \frac{3}{4}^+} \left(s-\frac{3}{4}\right) \Ks_{\frac{1}{2},4N}^+(0,0;2s) \right) \frac{\partial}{\partial s} \frac{2}{3} \frac{4^{1-s}\pi i^{-\frac{1}{2}}\Gamma(2s-1)}{\Gamma\left(s+\frac{1}{4}\right)\Gamma\left(s-\frac{1}{4}\right)}  v^{1-s-\frac{1}{4}} \Big\vert_{s=\frac{3}{4}} \\
= -\rf(N)\log(16v).
\end{multline*}
Furthermore, we have
\begin{align*}
\lim_{s \to \frac{3}{4}^+} \frac{2}{3} \frac{4^{1-s}\pi i^{-\frac{1}{2}}\Gamma(2s-1)}{\Gamma\left(s+\frac{1}{4}\right)\Gamma\left(s-\frac{1}{4}\right)}  v^{1-s-\frac{1}{4}} = \frac{2}{3}(1-i) \pi.
\end{align*}
and hence
\begin{multline*}
\hspace*{\leftmargini} \frac{\partial}{\partial s} \frac{2}{3} \left(s-\frac{3}{4}\right)\frac{4^{1-s}\pi i^{-\frac{1}{2}}\Gamma(2s-1)}{\Gamma\left(s+\frac{1}{4}\right)\Gamma\left(s-\frac{1}{4}\right)} \Ks_{\frac{1}{2},4N}^+(0,0;2s) v^{1-s-\frac{1}{4}} \Big\vert_{s=\frac{3}{4}} \\
= -\rf(N)\log(16v) + \frac{2}{3}(1-i) \pi \cf(0).
\end{multline*}

\item Suppose that $n > 0$. Using \eqref{eq:whittakerspectral}, we obtain
\begin{multline*}
\hspace*{\leftmargini} \left(\lim_{s \to \frac{3}{4}^+} \frac{2}{3} i^{-\frac{1}{2}} \frac{\pi^{s}v^{-\frac{1}{4}} n^{s-1}}{\Gamma\left(s+\frac{1}{4}\right)} W_{\frac{1}{4},s-\frac{1}{2}}(4\pi n v) e^{2\pi i n u} \right) \frac{\partial}{\partial s}  \left(s-\frac{3}{4}\right) \Ks_{\frac{1}{2},4N}^+(0,n;2s) \Big\vert_{s=\frac{3}{4}} \\
= \frac{2}{3}(1-i) \pi \cf(n) q^n.
\end{multline*}
If $n \neq m^2 \geq 1$, the residue of $\Ks_{\frac{1}{2},4N}^+(0,n;2s)$ vanishes by Proposition \ref{prop:Kloostermanzetaresidue}. If $n = m^2 \geq 1$, we employ \eqref{eq:whittakerspectral} and Lemma \ref{lem:preimageofincompletegamma}, getting
\begin{align*}
\hspace*{\leftmargini} \frac{\partial}{\partial s} \frac{2}{3} i^{-\frac{1}{2}} \frac{\pi^{s}v^{-\frac{1}{4}} n^{s-1}}{\Gamma\left(s+\frac{1}{4}\right)} W_{\frac{1}{4},s-\frac{1}{2}}(4\pi n v) e^{2\pi i n u} \Big\vert_{s=\frac{3}{4}}
= \frac{2}{3} (1-i) \pi  \left(\gamma + \log(\pi n) +  \alpha(4nv) \right) q^n.
\end{align*}
According to Proposition \ref{prop:Kloostermanzetaresidue}, we infer that
\begin{multline*}
\hspace*{\leftmargini} \frac{\partial}{\partial s} \left(s-\frac{3}{4}\right) \frac{2}{3} i^{-\frac{1}{2}} \pi^{s}v^{-\frac{1}{4}} \frac{\Ks_{\frac{1}{2},4N}^+(0,n;2s) n^{s-1}}{\Gamma\left(s+\frac{1}{4}\right)} W_{\frac{1}{4},s-\frac{1}{2}}(4\pi \vt{n} v) e^{2\pi i n u}  \Big\vert_{s=\frac{3}{4}} \\
= 2 \delta_{\square}(n) \rf(N) \left(\gamma + \log(\pi n) +  \alpha(4nv) \right) q^n + 
\frac{2}{3}(1-i) \pi \cf(n) q^n,
\end{multline*}
where $\delta_{\square}$ is the indicator function of $n$ being a square.

\item Suppose that $n < 0$. According to Proposition \ref{prop:Kloostermanzetaresidue}, $\Ks_{\frac{1}{2},4N}^+(0,n;2s)$ is regular at $s=\frac{3}{4}$ for every $n < 0$ and thus 
\begin{align*}
\hspace*{\leftmargini} \lim_{s \to \frac{3}{4}^+} \left(s-\frac{3}{4}\right) \Ks_{\frac{1}{2},4N}^+(0,n;2s) = 0
\end{align*}
for every $n < 0$. Utilizing \eqref{eq:whittakerspectral} again, we obtain
\begin{multline*}
\hspace*{\leftmargini} \frac{\partial}{\partial s} \frac{2}{3}\left(s-\frac{3}{4}\right)i^{-\frac{1}{2}}\pi^{s}v^{-\frac{1}{4}} \frac{\Ks_{\frac{1}{2},4N}^+(0,n;2s) n^{s-1}}{\Gamma\left(s-\frac{1}{4}\right)} W_{-\frac{1}{4},s-\frac{1}{2}}(4\pi \vt{n} v) e^{2\pi i n u} \Big\vert_{s=\frac{3}{4}} \\
= \frac{2}{3}\pi^{\frac{1}{2}} (1-i) \cf(n) \Gamma\left(\frac{1}{2},4\pi\vt{n}v\right) q^n.
\end{multline*}
\end{enumerate}
The claim follows by collecting the results of each item and rearranging them.
\end{proof}

\section{Proof of Theorem \ref{thm:maingeneral} \texorpdfstring{$\mathrm{(iii)}$}{(iii)}} \label{sec:proofof1.3part3}

\subsection{The form \texorpdfstring{$\xi_{\frac{1}{2}}\Gc$}{xi1/2G}}
To motivate the result of Proposition \ref{prop:Gshadow} below, we combine \eqref{eq:Deltasplitting}, \eqref{eq:DeltaGresidue} with Proposition \ref{prop:Eisensteinresidue} and the fact that 
\begin{align} \label{eq:xiHc}
\xi_{\frac{3}{2}}\Hc = -\frac{1}{16\pi} \theta,
\end{align}
see \cite{thebook}*{Theorem 6.3}). We obtain
\begin{align} \label{eq:DeltaG}
\xi_{\frac{3}{2}} \xi_{\frac{1}{2}} \Gc(\tau) = \frac{1}{2} f_{-1}(\tau) = \frac{\rf(N)}{2} \theta(\tau) = -\frac{16\pi}{2} \rf(N) \xi_{\frac{3}{2}} \Hc(\tau),
\end{align}
from which we in turn deduce that
\begin{align} \label{eq:gunknown}
\xi_{\frac{1}{2}} \Gc(\tau) = - 8 \pi \rf(N) \Hc(\tau) + g(\tau), \qquad g \in M_{\frac{3}{2}}^+(4N).
\end{align}
The following result computes the Fourier expansion of $g$ in terms of $\Ks_{\frac{1}{2},4N}^+$. In the upcoming subsections, we rewrite $g$ in terms of the generalized Cohen--Eisenstein series and the mock modular form $\Hs_{1,N}$, see \eqref{eq:peiandwangresult} below, which completes the proof of Theorem \ref{thm:maingeneral} (iii).

\begin{prop} \label{prop:Gshadow}
Let $N$ be odd and square-free. Then, we have
\begin{multline*} 
\xi_{\frac{1}{2}} \Gc(\tau) = \frac{1}{3} - \rf(N) v^{-\frac{1}{2}} - 2\sqrt{\pi} \rf(N) \sum_{m \geq 1} m \Gamma\left(-\frac{1}{2},4\pi m^2 v \right) q^{-m^2} \\
- \frac{4\pi(1+i)}{3} \sum_{\substack{n \geq 1 \\ n \equiv 0,3 \pmod*{4}}} \overline{\Ks_{\frac{1}{2},4N}^+\left(0,-n;\frac{3}{2}\right)} n^{\frac{1}{2}} q^{n}.
\end{multline*}
Equivalently, the modular form $g \in M_{\frac{3}{2}}^+\left(4N\right)$ in \eqref{eq:gunknown} has the Fourier expansion
\begin{align*}
g(\tau) = \frac{1}{3} -\frac{2}{3}\pi \rf(N) + \frac{1}{3} \sum_{\substack{n \geq 1 \\ n \equiv 0,3 \pmod*{4}}} \left(24 \pi \rf(N)H_{1,1}(n) - 4\pi(1+i)\overline{\Ks_{\frac{1}{2},4N}^+\left(0,-n;\frac{3}{2}\right)} n^{\frac{1}{2}}\right)q^n.
\end{align*}
\end{prop}

\begin{proof}
We compute $\xi_{\frac{1}{2}}\Gc$ by inspecting the action of $\xi_{\frac{1}{2}}$ on each term Theorem \ref{thm:maingeneral} (ii).
\begin{enumerate}
\item Clearly, we have $\overline{\rf(N)} = \rf(N)$ and
\begin{align*}
\xi_{\frac{1}{2}} \frac{2}{3} v^{\frac{1}{2}} = \frac{1}{3}, \qquad \xi_{\frac{1}{2}} \log(16v) = v^{-\frac{1}{2}}.
\end{align*}

\item Supppose that $n \geq 0$. By holomorphicity, we deduce that
\begin{align*}
\xi_{\frac{1}{2}} \sum_{\substack{n \geq 0 \\ n \equiv 0,1 \pmod*{4}}} \cf(n)q^n = \xi_{\frac{1}{2}} \sum_{m \geq 1 } \rf(N) \left(\gamma + \log\left(\pi m^2\right) \right) q^{m^2} = 0.
\end{align*}
Lemma \ref{lem:preimageofincompletegamma} yields
\begin{align*}
\xi_{\frac{1}{2}} 2 \rf(N) \sum_{m \geq 1 }  \alpha\left(4m^2v\right) q^{m^2} = -2 \pi^{\frac{1}{2}} \rf(N) \sum_{m \geq 1} m \Gamma\left(-\frac{1}{2}, 4\pi m^2v\right)q^{-m^2}.
\end{align*}

\item Let $n < 0$ and recall \eqref{eq:KloostermanZetaResidueRegular}. A standard computation shows that
\begin{align*}
\xi_{\frac{1}{2}} \left[\Gamma\left(\frac{1}{2},4\pi\vt{n}v\right) q^n \right] = -(4\pi \vt{n})^{\frac{1}{2}} q^{-n},
\end{align*}
see the proof of \cite{thebook}*{Theorem 5.10} for example. Consequently, we obtain
\begin{multline*}
\hspace*{\leftmargini} \xi_{\frac{1}{2}} \frac{2}{3} (1-i) \pi^{\frac{1}{2}} \sum_{\substack{n < 0 \\ n \equiv 0,1 \pmod*{4}}}  \cf(n) \Gamma\left(\frac{1}{2},4\pi\vt{n}v\right) q^n \\
= \frac{2}{3} (1+i) \pi^{\frac{1}{2}} \sum_{\substack{n < 0 \\ n \equiv 0,1 \pmod*{4}}} \overline{ \Ks_{\frac{1}{2},4N}^+\left(0,n;\frac{3}{2}\right)} \xi_{\frac{1}{2}} \Gamma\left(\frac{1}{2},4\pi\vt{n}v\right) q^n \\
= -\frac{4\pi(1+i)}{3} \sum_{\substack{n > 0 \\ n \equiv 0,3 \pmod*{4}}} \overline{\Ks_{\frac{1}{2},4N}^+\left(0,-n;\frac{3}{2}\right)} n^{\frac{1}{2}} q^n.
\end{multline*}
\end{enumerate}
This proves the first assertion. We conclude the second assertion by the first assertion, \eqref{eq:gunknown} and the definition of $\Hc$ in \eqref{eq:Hcdef}.
\end{proof}

\subsection{Connection to generalized Hurwitz class numbers}

The numbers $A(2,n)$ and $A(p,n)$ from Proposition \ref{prop:localfactors} appear in Pei and Wang \cite{peiwang} too. Let $r \in \N$ and $p > 2$ be a prime. Pei and Wang \cite{peiwang}*{p.\ $106$} define the numbers
\begin{multline*}
\Ac_r(2,n) \coloneqq 2^{-(2r+1)}\left(1+(-1)^{r}i\right) \\
\times \begin{cases}
\frac{1-2^{(1-2r)\frac{\nu_2(n)-1}{2}}}{1-2^{1-2r}} - 2^{(1-2r)\frac{\nu_2(n)-1}{2}} & \text{if } 2 \nmid \nu_2(n), \\
\frac{1-2^{(1-2r)\frac{\nu_2(n)}{2}}}{1-2^{1-2r}} - 2^{(1-2r)\frac{\nu_2(n)}{2}} & \text{if } 2 \mid \nu_2(n), \frac{(-1)^{r}n}{2^{\nu_2(n)}} \equiv -1 \pmod*{4}, \\
\frac{1-2^{(1-2r)\frac{\nu_2(n)}{2}}}{1-2^{1-2r}} + 2^{(1-2r)\frac{\nu_2(n)}{2}}\left(1+2^{1-r}\left(\frac{\frac{(-1)^{r}n}{2^{\nu_2(n)}}}{2}\right)\right) & \text{if } 2 \mid \nu_2(n), \frac{(-1)^{r}n}{2^{\nu_2(n)}} \equiv 1 \pmod*{4},
\end{cases}
\end{multline*}
as well as
\begin{align*}
\Ac_r(p,n) \coloneqq \begin{cases}
\frac{(p-1)\left(1-p^{(1-2r)\frac{\nu_p(n)-1}{2}}\right)}{p\left(p^{2r-1}-1\right)} - p^{(1-2r)\frac{\nu_p(n)+1}{2}-1} & \text{if } 2 \nmid \nu_p(n), \\
\frac{(p-1)\left(1-p^{(1-2r)\frac{\nu_p(n)}{2}}\right)}{p\left(p^{2r-1}-1\right)} + \left(\frac{\frac{(-1)^{r}n}{p^{\nu_p(n)}}}{p}\right) p^{(1-2r)\frac{\nu_p(n)+1}{2}-\frac{1}{2}} & \text{if } 2 \mid \nu_p(n).
\end{cases}
\end{align*}

\begin{rmk}
We believe that there is a small typo in the definition of $\Ac_{r}(p,n)$ in \cite{peiwang}. Namely, if $2 \mid \nu_p(n)$, the sign of the second term in their definition of $\Ac_r(p,n)$ should be flipped. We adjusted our definition of their numbers accordingly.
\end{rmk}

If $r = 1$, these numbers pertain to the case of weight $\frac{3}{2}$, while our local factors from Propositions \ref{prop:Kloostermanzetageneral}, \ref{prop:localfactors} arise from a weight $\frac{1}{2}$ automorphic form. However, we expect a connection between both weights by \eqref{eq:Kloostermanfunctionalequation}.
\begin{lemma} \label{lem:peiwangconnection}
Suppose that $N$ is odd and square-free.
\begin{enumerate}[label={\rm (\roman*)}]
\item We have
\begin{align*}
A(2,n) = \overline{\mathcal{A}_1(2,-n)}, \qquad A(p,n) = \overline{\mathcal{A}_1(p,-n)} = \mathcal{A}_1(p,-n).
\end{align*}
\item We have
\begin{align*}
\prod_{\substack{p \text{ prime} \\ p \mid N}} \Ac_1(p,n) = \frac{1}{N} \sum_{\ell \mid N} \ell \prod_{\substack{p \text{ prime} \\ p \mid \ell}} \left(\Ac_1(p,n)-\frac{1}{p}\right).
\end{align*}
\end{enumerate}
\end{lemma}

\begin{proof}
\
\begin{enumerate}[label={\rm (\roman*)}]
\item This follows either by a case-by-case inspection of both sides of the claim or by combining \eqref{eq:Kloostermanfunctionalequation} with Proposition \ref{prop:Kloostermanzetageneral}.
\item We have
\begin{align*}
\prod_{\substack{p \text{ prime} \\ p \mid N}} \Ac_1(p,n) = \prod_{\substack{p \text{ prime} \\ p \mid N}} \left(\Ac_1(p,n)-\frac{1}{p}+\frac{1}{p}\right) = \sum_{\ell \mid N} \prod_{\substack{p \text{ prime} \\ p \mid \ell}} \left(\Ac_1(p,n)-\frac{1}{p}\right) \prod_{\substack{p \text{ prime} \\ p \mid \frac{N}{\ell}}} \frac{1}{p},
\end{align*}
which implies the claim directly. Alternatively, this can be deduced from an identity at the bottom of \cite{peiwangbook}*{p.\ 239}. \qedhere
\end{enumerate}
\end{proof}

In the course of proving their Theorem 1 (I), Pei and Wang establish the following identity.
\begin{lemma}[\protect{\cite{peiwang}}] \label{lem:Hurwitzrewrite}
Let $N > 1$ be odd and square-free. Let $\ell \mid N$. Let $n \geq 1$ satisfy $n \equiv 0, 3 \pmod*{4}$. Then, we have
\begin{multline*}
\hspace*{\leftmargini} H_{\ell,N}(n) = \frac{4\pi(1+i)}{12} \prod_{\substack{p \text{ prime} \\ p \mid \ell}} (1-p) \\
\times \frac{L_{4 N} (1, \chi_{t})}{L_{4N} (2, \id)} T_{4N,0}^{\chi_{t}}(m) \Bigg(\prod_{\substack{p \text{ prime} \\ p \mid \ell}} \left(\Ac_1(p,n) - \frac{1}{p}\right) \Bigg) \left(\Ac_1(2,n) + \frac{1-i}{8} \right)n^\frac{1}{2}.
\end{multline*}
\end{lemma}

\begin{proof}
Compare \cite{peiwang}*{(4)} with \cite{peiwang}*{p.\ 111} and the definition of the numbers $H_{\ell,N}(n)$ in \eqref{eq:HNNdef}, \eqref{eq:HellNdef}. Their proof works for the case $\ell = 1$ too.
\end{proof}

\subsection{Proof of Theorem \ref{thm:maingeneral} (iii)}
Recall Theorem \ref{thm:maingeneral} (i), (ii) and Proposition \ref{prop:Gshadow}.

\begin{proof}[Proof of Theorem \ref{thm:maingeneral} (iii)]
Let $n \geq 1$ satisfy $n \equiv 0, 3 \pmod*{4}$. To prove the claimed expression for $\xi_{\frac{1}{2}}\Gc$, we use Proposition \ref{prop:Kloostermanzetageneral} and have ($-n = tm^2$ with $t$ square-free)
\begin{align*}
\overline{\Ks_{\frac{1}{2},4N}^+\left(0,-n;\frac{3}{2}\right)}
= \overline{\frac{L_{4 N} (1, \chi_{t})}{L_{4N} (2, \id)} T_{4N,0}^{\chi_{t}}(m)
\Bigg(\prod_{\substack{p \text{ prime} \\ p \mid N}} A(p,-n) \Bigg) \cdot \left(  A(2,-n) + \frac{a(4,-n)}{8} \right)}.
\end{align*}
Since the input $s$ as well as all the characters are real, the Dirichlet $L$-functions involved are real-valued. For the same reasons, the function $T_{4N,0}^{\chi_{t}}$ is real-valued. According to Lemma \ref{lem:peiwangconnection} (i) and \eqref{eq:localoddcontribution}, we hence obtain
\begin{align*}
\overline{\Ks_{\frac{1}{2},4N}^+\left(0,-n;\frac{3}{2}\right)} = \frac{L_{4 N} (1, \chi_{t})}{L_{4N} (2, \id)} T_{4N,0}^{\chi_{t}}(m)
\Bigg(\prod_{\substack{p \text{ prime} \\ p \mid N}} \Ac_1(p,n) \Bigg) \left(\Ac_1(2,n) + \frac{1-i}{8} \right).
\end{align*}
We employ Lemmas \ref{lem:peiwangconnection} (ii) and \ref{lem:Hurwitzrewrite}, which yields
\begin{align*}
&4\pi(1+i)\overline{\Ks_{\frac{1}{2},4N}^+\left(0,-n;\frac{3}{2}\right)} n^{\frac{1}{2}} \\
&= 4\pi(1+i)\frac{L_{4 N} (1, \chi_{t})}{L_{4N} (2, \id)} T_{4N,0}^{\chi_{t}}(m)
\Bigg(\frac{1}{N} \sum_{\ell \mid N} \ell \prod_{\substack{p \text{ prime} \\ p \mid \ell}} \left(\Ac_1(p,n)-\frac{1}{p}\right)\Bigg) \left(\Ac_1(2,n) + \frac{1-i}{8} \right) n^{\frac{1}{2}} \\
&= \frac{12}{N} \sum_{\ell \mid N} \ell H_{\ell,N}(n) \prod_{\substack{p \text{ prime} \\ p \mid \ell}} \frac{1}{1-p}.
\end{align*}
We infer that the function $g$ from \eqref{eq:gunknown} and Proposition \ref{prop:Gshadow} can be written as
\begin{align*}
g(\tau) = 4 \prod_{\substack{p \text{ prime} \\ p \mid N}} \frac{1}{p+1} \sum_{n \geq 0} H_{1,1}(n)q^n - \frac{4}{N} \sum_{\ell \mid N} \ell \prod_{\substack{p \text{ prime} \\ p \mid \ell}} \frac{1}{1-p} \sum_{n \geq 1} H_{\ell,N}(n)q^n + \frac{1}{3}.
\end{align*}
We add and subtract $H_{N,N}(0) = L_N(-1,\mathrm{id})$, arriving at
\begin{align} \label{eq:peiandwangresult}
g(\tau) = - \frac{4}{N} \sum_{\substack{\ell \mid N \\ \ell > 1}} \ell \prod_{\substack{p \text{ prime} \\ p \mid \ell}} \frac{1}{1-p} \Hs_{\ell,N}(\tau) + 4 \prod_{\substack{p \text{ prime} \\ p \mid N}} \frac{1}{p+1} \sum_{n \geq 0} H_{1,1}(n)q^n  - \frac{4}{N} \sum_{n \geq 1} H_{1,N}(n)q^n.
\end{align}
This proves the claimed representation of $\xi_{\frac{1}{2}}\Gc$. The second assertion follows immediately from Theorem \ref{thm:maingeneral} (i) and the properties of $\xi_{\frac{1}{2}}$. The third assertion follows from the first one by \eqref{eq:xiHc}.
\end{proof}

\section{Proof of Theorems \ref{thm:mainholomorphic} and \ref{thm:mainnonholomorphic}} \label{sec:proofof1.1and1.2}

The representation of $\xi_{\frac{1}{2}}\Gc$ in Theorem \ref{thm:maingeneral} (iii) isolates the higher level holomorphic modular correction of the mock modular generating function of $H_{1,1}(n)$. In addition to that, it isolates the multiple of the non-holomorphic part of $\Hc$ that completes the mock modular generating function of $H_{1,N}(n)$ to a higher level analog of $\Hc$. We make this precise in the proofs below.
\begin{proof}[Proof of Theorem \ref{thm:mainholomorphic}]
Recall that $g$ appearing in \eqref{eq:gunknown}, \eqref{eq:peiandwangresult} and Proposition \ref{prop:Gshadow} is modular and that Pei and Wang's forms $\Hs_{\ell,N}$ for $1 < \ell \mid N$ are modular on their own. By virtue of \eqref{eq:peiandwangresult}, we deduce that the linear combination 
\begin{align*}
\prod_{\substack{p \text{ prime} \\ p \mid N}} \frac{1}{p+1} \sum_{n \geq 0} H_{1,1}(n)q^n  - \frac{1}{N} \sum_{n \geq 1} H_{1,N}(n)q^n
\end{align*}
is modular as well. Holomorphicity and the growth condition are immediate.

The second assertion follows by the fact that if $N=p$ is an odd prime then the space $E_{\frac{3}{2}}^+(4p)$ is spanned by $\Hs_{p,p}(\tau)$. Consequently, we must have
\begin{align*}
g(\tau) = \frac{\frac{1}{3}\left(1-\frac{1}{p+1}\right)}{-\frac{1}{12} (1-p)} \Hs_{p,p}(\tau).
\end{align*}
Equating this representation of $g$ with \eqref{eq:peiandwangresult} and rearranging concludes the proof.
\end{proof}

For the sake of completeness, we provide two proofs of Theorem \ref{thm:mainnonholomorphic}.
\begin{proof}[Proof of Theorem \ref{thm:mainnonholomorphic}]
The first assertion follows by inserting $\Hc(\tau)$ from \eqref{eq:Hcdef} into Theorem \ref{thm:maingeneral} (iii). Alternatively, the modularity of the form in Theorem \ref{thm:mainholomorphic} implies the modularity of the form in Theorem \ref{thm:mainnonholomorphic} and vice versa (as pointed out in the remark after Theorem \ref{thm:mainnonholomorphic}), while harmonicity and growth conditions are immediate.
The second assertion follows by Proposition \ref{prop:Eisensteinresidue} and \eqref{eq:xiHc}.
\end{proof}

\section{Further questions}
We conclude with some questions for future work. The first was suggested by Michael Griffin, the fourth was suggested by Larry Rolen, and the fifth was suggested by Nikolaos Diamantis.
\begin{enumerate}
\item If we multiply the form in Theorem \ref{thm:mainnonholomorphic} by $\theta$ and apply the holomorphic projection operator along the lines of Mertens' work \cites{mertens14, mertens16}, do we obtain interesting higher level class number relations?
\item What are the Shimura lifts of the modular form in Theorem \ref{thm:mainholomorphic}? Compare \cite{peiwang}*{p.\ 116} for the Shimura lifts of $\Hs_{p,p}$, where $p$ is an odd prime. Are there preimages of the form from Theorem \ref{thm:maingeneral} under the Shintani lifts, e.g.\ along the lines of \cite{alneschw21}*{Remark 1.2}?
\item Are there relations between our coefficients in Theorem \ref{thm:mainholomorphic} and the representation numbers in \cite{luozhou}?
\item Are there interesting congruences satisfied by the coefficients in Theorem \ref{thm:mainholomorphic}?
\item Is there a relation between the non-critical (modular) $L$-values and the generalized Hurwitz class numbers via the sesquiharmonic forms underpinning both?
\item Are there versions of these results using the $\eta$-multiplier? The analogs would involve lifts of $\eta(\tau)$ and Andrews' mock modular generating function of the $\mathrm{spt}(n)$ function \cites{andrews, bringmann08, folono}.
\end{enumerate}

\end{document}